\documentclass[11pt,reqno]{amsart}
\usepackage{mathrsfs}
\usepackage{amsgen}
\usepackage{amscd}
\usepackage{amsmath}
\usepackage{latexsym}
\usepackage{amsfonts}
\usepackage{amssymb}
\usepackage{amsthm}
\usepackage{graphicx}
\usepackage{color}
\usepackage{amsthm,amssymb}
\usepackage{graphicx}
\usepackage{enumerate}

\usepackage{amssymb,amsmath,graphicx,amsfonts,euscript}
\usepackage{color}
\usepackage{mathrsfs}
\usepackage{diagbox}
\usepackage{booktabs}
\usepackage{float}
\usepackage[colorlinks = true,
linkcolor = blue,
urlcolor  = blue,
citecolor = red,
anchorcolor = blue,
backref]{hyperref}

\usepackage[margin=3cm, a4paper]{geometry}

\def\H{{\cal H}}
\def\F{\mathscr{F} }
\def\N{\mathbb{N}}
\def\R{\mathbb{R}}

\def\C{\mathbb{C}}

\def\H2{H^2(\R^N)}
\def\L2{L^2(\R^N)}
\def\to{\rightarrow}

\def\jb#1{\langle#1\rangle}
\def\norm#1{\left\|#1\right\|}

\def\H{{\cal H}}

\def\H1{H^1(\R)}

\def\sc#1{\norm{#1}_{S_{s_c}}}
\def\so#1{\norm{#1}_{S_{0}}}
\def\xsc#1{\norm{#1}_{X_{s_c}}}
\def\xo#1{\norm{#1}_{X_{0}}}
\def\ysc#1{\norm{#1}_{Y_{s_c}}}
\def\yo#1{\norm{#1}_{Y_{0}}}
\def\zsc#1{\norm{#1}_{Z_{s_c}}}
\def\zo#1{\norm{#1}_{Z_{0}}}

\def\nh#1#2{\left\|#1\right\|_{\dot H_x^{#2}}}
\def\nl#1#2{\left\|#1\right\|_{L_x^{#2}}}

\def\ntx#1#2#3{\left\|#1\right\|_{L^#2_tL^#3_x}}
\def\nth#1#2#3{\left\|#1\right\|_{L^#2_t\dot H_x^{#3}}}
\def\ltx#1#2{\left\|#1\right\|_{L_{tx}^#2}}

\def\ei#1{e^{#1 i\frac{t}{\ve^2}}}
\def\oo{\infty}
\def\jd#1{\jb{\na}^{#1}}
\def\naabs#1{|\na|^{#1}}

\newcommand{\les}{\lesssim}
\newcommand{\ges}{\gtrsim}
\newcommand{\wt}{\widetilde}

\newcommand{\al}{\alpha}

\newcommand{\ga}{\gamma}
\newcommand{\de}{\delta}
\newcommand{\ve}{\varepsilon}

\newcommand{\De}{\Delta}

\newcommand{\p}{\partial}
\newcommand{\na}{\nabla}
\newcommand{\re}{\mathop{\mathrm{Re}}}
\newcommand{\im}{\mathop{\mathrm{Im}}}

\newcommand{\ti}{\widetilde}
\newcommand{\ba}{\overline}

\newcommand{\eq}[1]{ 
    \begin{equation}
        \left\lbrace 
        \begin{aligned}#1
        \end{aligned}
        \right.
    \end{equation}
}
\newcommand{\eqn}[1]{ 
    \begin{equation*}
        \left\lbrace 
        \begin{aligned}#1
        \end{aligned}
        \right.
    \end{equation*}
}

\newcommand{\Del}[1]{}

\numberwithin{equation}{section}

\newtheorem{thm}{Theorem}[section]

\newtheorem{lem}[thm]{Lemma}

\newtheorem{prop}[thm]{Proposition}
\newtheorem{definition}[thm]{Definition}

\theoremstyle{remark}
\newtheorem{remark}[thm]{Remark}
\newtheorem*{exam*}{Examples}

\begin{document}

    \setcounter{page}{1}

    \title[High-order Asymptotic expansion for KG]{High-order asymptotic expansion for the nonlinear Klein-Gordon equation in the non-relativistic limit regime}

\author{Jia Shen}
\address{Jia Shen \newline School of Mathematical Sciences and LPMC\\
Nankai University\\
Tianjin 300071, China}
\email{shenjia@nankai.edu.cn}

\author{Yanni Wang}
\address{Yanni Wang \newline Center for Applied Mathematics\\
Tianjin University,
Tianjin 300072, China}
\email{wyn{\_}01@tju.edu.cn}

\author{Haohao Zheng}
\address{Haohao Zheng \newline Center for Applied Mathematics\\
Tianjin University,
Tianjin 300072, China}
\email{hhzheng@tju.edu.cn}



  \begin{abstract}
      This paper presents an investigation into the high-order asymptotic expansion for 2D and 3D cubic nonlinear Klein-Gordon equations in the non-relativistic limit regime.   
      There are extensive numerical and analytic results concerning that the solution of NLKG can be approximated by first-order modulated Schr\"odinger profiles in terms of $\ei{}v + c.c. $, where $v$ is the solution of related NLS
      and ``$c.c.$" denotes the complex conjugate.
      Particularly, the best analytic result up to now is given in \cite{lei}, which proves that the $L_x^2$ norm of the error can be controlled by 
      $\ve^2 +(\ve^2t)^{\frac \al 4}$
      for $H^\al_x$-data, $\al \in [1, 4]$. 
      As for the high-order expansion, to our best knowledge, there are only numerical results, while the theoretical one is lacking.    
      
      In this paper, we extend this study further 
      and give the first high-order analytic result. We introduce the high-order expansion inspired by the numerical experiments in \cite{schratz2020, faou2014a}:
      \[ 
      \ei{}v
      +\ve^2
      \Big( \frac 18 \ei{3} v^3  +\ei{} w
      \Big)
      +c.c.,
      \] 
      where $w$ is the solution to some specific Schr\"odinger-type equation.
      We show that the $L_x^2$ estimate of the error
      is of higher order $\ve^4+\left(\ve^2t\right)^\frac{\al}{4}$ for $H^\al_x$-data, $\al \in [4, 8]$.
      Besides, some counter-examples are given to suggest the sharpness of this upper bound.      
  \end{abstract}
  
  \maketitle
  
  \renewcommand{\theequation}{\thesection.\arabic{equation}}
  \setcounter{equation}{0}

  \section{Introduction}
  In this paper, we consider the following cubic nonlinear Klein-Gordon(NLKG) equation in the non-relativistic limit regime for dimension $d=2, 3$,
  \eq{\label{eq:u} 
      &\ve^2 \p_{tt} u^{\ve}-\De u^{\ve} + \frac{1}{\ve^2} u^{\ve} = -|u^{\ve}|^2u^{\ve}, \\
      &u^{\ve}(0)=u_0, \quad \p_t u^{\ve}(0)=\frac{u_1}{\ve^2},}
  where the solution $u^\ve (t,x): \R\times\R^d \to \R$
  and initial data $u_0, u_1$ are real-valued and independent of $\ve$.
  Here $\ve$ is inversely proportional to the speed of light
  and the non-relativistic limit regime means
  the regime when the speed of light goes to infinity or equivalently $\ve $ approaches 0.
  The NLKG equation in the non-relativistic limit serves as a mathematical model for 
  describing a free particle with spin zero. 
  This equation also can be regarded as the relativistic counterpart of the Schr\"odinger equation, 
  specifically designed to characterize the kinetic energy component of scalar particles. 
  The equation \eqref{eq:u} conserves the energy
  \begin{align*}
      E(t)
      &:=
      \int_{\R^d}
      \left(\ve^{2}
      \left|\p_{t}u^{\ve}\right|^2
      +\left|\na u^{\ve}\right|^2 
      +\ve^{-2}\left|u^{\ve}\right|^2
      +\frac{1}{2}\left|u^{\ve}\right|^4
      \right)dx\\
      &=\int_{\R^d}
      \Big[
      \ve^{-2}
      \left(
      \left| u_0\right|^2
      +\left| u_1\right|^2
      \right)
      +\left|\na u_0\right|^2 
      +\frac{1}{2}\left|u_0\right|^4
      \Big]dx =E(0).
  \end{align*}
  
  The NLKG equation \eqref{eq:u} in the non-relativistic regime has been extensively studied in both mathematical and physical fields.
  It is well-known that the NLKG equation reduces to the nonlinear Schr\"odinger(NLS) equation in this limit, 
  providing a mathematical tool for bridging the gap between relativistic and non-relativistic regimes.
  Here, we highlight a crucial observation that the solution exhibits high-frequency oscillations over time within this regime. 
  The presence of oscillations makes classical integration schemes invalid and gives rise to challenges in numerical analysis. 
  Furthermore, complications emerge in analysis wherein solutions may not adhere to the uniform $H^1_x$ bound as the energy lacks uniform boundedness when $\varepsilon $ approaches 0.
  
  \subsection{High-order modulated expansion to Schr\"odinger profiles}
  Firstly, we introduce the \emph{modulated Fourier expansion} also as a WKB expansion to approximate the solution $u^\varepsilon$ in terms of the modulated profiles.
  The modulated Fourier expansion has been extensively utilized in the numerical analysis of oscillatory problems \cite{cohen2003, faou2013, hairer2006}.
  The basic idea is to expand solution $u^\ve$ as follows 
  \begin{align*}
      u^\ve(t)=\sum_{m\in \N_+}e^{i\frac{mt}{\ve^2}} u_m (t), 
  \end{align*} 
  where all time derivatives of profiles $u _m(t)$ is uniformly bounded as $\ve $ approaches to $ 0$.
  In particular, these modulated profiles are typically selected as solutions to the cubic NLS and thus called \emph{Schr\"odinger profiles}. 
  We note that those profiles $u_m(x)$ are also allowed to be $\ve $-dependent oscillating
  and the corresponding expansion is called \emph{multiscale frequency expansion}.
  We will see the expansion involving \emph{Schr\"odinger-wave profiles} below is an example of this type of expansion.
  However, we focus on the modulated Fourier expansion involving Schr\"odinger profiles in this paper.
  
  In this paper, we introduce a high-order modulated Fourier expansion inspired by the numerical study in \cite{schratz2020, faou2014a}:
  \begin{definition}[Asymptotic expansion]\label{defn;main} We make the following notation.
      \begin{enumerate}
          \item (Schr\"odinger profiles)
          Let the profiles $v$, $w$ be defined as the solutions of the following nonlinear Schr\"odinger equations, respectively:
          \eq{\label{eq:v}
              & 2i \p_t v-\De v +3|v|^2v=0, \\
              & v(0) = v_0=\frac12(u_0-iu_1), 
          }
          and 
          \eq{ 
              \label{eq:f}
              &2i\p_t w -\De w +\p_{tt} v +3 \left(\frac 18 |v|^4v+v^2 \ba w+ 2|v|^2w\right)=0,\\
              & w_0
              =\frac 14 \De (v_0-\ba {v_0})
              -\frac14 v_0^3+\frac 18 \ba {v_0}^3 
              -\frac{3}{4}|{v_0}|^2(v_0-\ba {v_0}).
          }
          \item (Leading terms) We defined the first-order and the next-order leading terms $\Phi_1$ and $\Phi_2$ by
          \begin{equation}
              \label{Phi_1}
              {\Phi_1}:= \ei{}v+c.c.,
          \end{equation}
          and
          \begin{equation}
              \Phi_2 :=\frac 18 \ei{3} v^3 +\ei{} w+c.c., 
              \label{Phi_2}
          \end{equation}
          where $v$ and $w$ solve \eqref{eq:v} and \eqref{eq:f}, respectively.
          \item (High-order expansion) Now, we introduce the desired expansion by
          \begin{align}
              \label{expansion of u}
              u^\ve
              &=\Phi_1 +\ve^2 \Phi_2+r_1, 
          \end{align}
          where $r_1$ denotes the related higher-order remainder.
      \end{enumerate}
  \end{definition}
  
  First, we recall the first-order modulated Fourier expansion that appeared in the previous literature \cite{masmoudi,bao2012b,faou2014} using the above notation:
  \begin{equation}
      \label{first order expansion of u}
      u^{\ve}= {\Phi_1}+r,
  \end{equation}  
  where $\Phi_1$ is given in \eqref{Phi_1}, and $r$ stands for the first-order remainder term.  
  Recall the result in \cite{lei}, the first-order remainder term $r$ in \eqref{first order expansion of u} verifies 
  \eq{
      \label{eq:r}
      &\ve^2 \p_{tt} r -\De r +\frac{1}{\ve^2} r
      + F_1(v, r)+F_2(v)+F_3(v)=0, \\
      &r(0)=0, \quad  \p_t r(0) = -(\p_t v(0)+\p_t \ba v(0)),	 
  }
  where 
\begin{align*}
    F_1(v, r): =3{\Phi_1}^2 r+3{\Phi_1}r^2+r^3,
    \quad
    F_2(v):= \ei{3}v^3+c.c.,
    \quad
    F_3(v):=\ve^2 
    \ei{ }\p_{tt}v+ c.c..
\end{align*}
  We note also that a choice of the equation of $v$ is the
  following Schr\"odinger-wave equation:
  \[ \ve^2\p_{tt}v+2i \p_t v-\De v +3|v|^2v=0, \]
  and the corresponding profile is called Schr\"odinger-wave profile,
  which we do not discuss further here.
  
  Numerous contributions in numerical analysis have extensively investigated the convergence rate of the first-order remainder term $r$ as $\ve$ approaches 0. 
  Pertinent studies, such as those conducted by Bao et al. \cite{bao2014a, bao2012a, bao2019}, Chartier et al. \cite{chartier2014b, chartier2015}, Dong \cite{dong2014}, Schratz \cite{schratz2020}, and Zhao \cite{zhao2017}, offer valuable insights into this convergence phenomenon.
  For example, the literature \cite{bao2019, chartier2015, schratz2020}
  empirically observes quadratic convergence in $\ve $ for the smooth initial data.
  Numerical experiments presented in  \cite{bao2016,bao2019,schratz2020} suggest that $r$ is of order $(1 + t)\ve^2$ for $H^3_x$-data, and $(1 + t)\ve$ for $H^2_x$-data. 
  These numerical experiments indicate that the convergence rate is not uniform over time and exhibits correlations with the regularity of the initial data.
  In addition,
  numerous studies have explored modulated Schr\"odinger-wave profiles, see \cite{bao2012b, bao2014, bao2019, schratz2020}.
  
  Some rigorous results regarding the convergence rate of the first-order remainder term $r$ have been established in the realm of analysis. In \cite{masmoudi2002}, Masmoudi and Nakanishi derived a local error estimate of $\ve^{\frac 12}$ in the $L^2$ framework for $O(1)$ time, assuming finite energy and $L^2$ convergence of the initial data. Furthermore, Bao, Lu, and Zhang \cite{bao2023} proved that for any $(u_0, u_1) \in \left( H^{s+4\al} \cap L^1\right)^2$ with $s >\frac 32$ and $d=3$, there holds  for time up to $O\left(\ve^{-\frac 1 2 \al}\right)$
  that
  $$\norm{u^\ve -\Phi_1}_{H^s_x}\les (1 + t)\ve ^ \al, \quad \al = 1, 2. $$
  Recently, Lei and Wu \cite{lei} significantly enhanced this result by proving that
  for $u_0, u_1 \in H^\al (\R ^d )$, $d=2,3$ with $ 1\le \al \le 4$ that
  for all $t\ge 0$,
  \[ \nl{u^\ve -\Phi_1}{2} \les \ve^2 +(\ve^2t)^{\frac 14\al}.\]
  They also showed the sharpness of the upper bounds for this result by constructing counter-examples.
  This enhanced result solidifies the validity of the aforementioned numerical experiments from an analytical perspective.
  For analytical results of the Schr\"odinger-waves profiles, we refer to \cite{lei}  and references therein.
  
  Next, we consider the results concerning the high-order expansions.
  Faou and Schratz \cite{faou2014a} presented explicit formulations for the first- and second-order terms in the high-order asymptotic expansion of the solution to \eqref{eq:u}.
  In \cite{schratz2020}, Schratz and Zhao provided the numerical analysis yielding an $O(\ve^4)$ error estimate in $H^1_x$ norm based on three forms of the high-order versions of asymptotic expansion: modulated Fourier expansion for $H^8_x$-data, multiscale expansion by frequency $H^4_x$-data, and Chapman-Enskog type expansion for $H^4_x$-data. For the analytical study, we use a variant version of the modulated Fourier expansion \eqref{expansion of u}. Recall their modulated Fourier expansion in \cite{schratz2020}, using our notation, is given as
\begin{align*}
         u^{\ve}(t) =  \ei{}v + 
         \ve^2\Big(
          \frac 18 \ei{3} v^3 +\ei{} w^* 
       - \frac34 \ei{} |v|^2  v
      \Big)+c.c. + r_1^*, 
\end{align*}
  where $w^*$ solves some Schr\"odinger type equation and $r^*_1$ is the high-order error.
  Here we note  the term $ 
        -\frac34 \ei{} |v|^2  v$ is unnecessary,
        as it can be incorporated into $\ei{}w^*$,
        thereby yielding our expansion \eqref{expansion of u}.
 
  We explain roughly our choice of $\Phi_2$ in \eqref{Phi_2} and the equation \eqref{eq:f} of $w$.
  A more detailed analysis and explanation 
  will be given in Section \ref{the high-order asymptotic expansion}. 
  In fact, the separation of the next-order term $\Phi_2$ is based on the error estimate for the first-order asymptotic expansion established in \cite{lei}.
According to the results in \cite{lei} that $r$ satisfying \eqref{eq:r} is of order $O(\ve^2)$, we decompose $r$ further by letting
  \begin{align*}
     r: = \ve^2 \Phi_2 +r_1,
  \end{align*}
  where  we expect that $\ve^2 \Phi_2$ with $\Phi_2\sim O(1)$ is the $O(\ve^2)$ leading term, and $r_1$ is the corresponding high-order remainder term.
Inserting this formula into \eqref{eq:r}, we have the equation of $r_1$
\begin{align*}
    \ve^2 \p_{tt} r_1 -\De r_1 +\frac{1}{\ve^2} r_1 +F_1(\Phi_1, \Phi_2, v)
    + F_2(\Phi_1, \Phi_2, r_1)
  =0,
\end{align*} 
where
\begin{align*}
    F_1(\Phi_1, \Phi_2, v)
    &=\ve^4 \p_{tt}\Phi_2 -\ve^2 \De \Phi_2 +\Phi_2
   + 3\ve^2 \Phi_1^2\Phi_2
   +F_2(v)
   +F_3(v),
   \notag\\
  F_2(\Phi_1, \Phi_2, r_1)
    &=3\ve^4 {\Phi_1}\Phi_2^2 +\ve^6 \Phi_2^3
    +
    3\left({\Phi_1}+\ve^2 {\Phi_2}\right)^2 {r_1}
    + 3\left({\Phi_1}+\ve^2 {\Phi_2}\right) r_1^2
    +r_1^3.
\end{align*}
   To show $r_1$ is of a higher order than $O(\ve^2)$, we need to apply the following principle which can be extracted from \cite{lei}: for a general equation
  \[
    \ve^2\p_{tt} r-\De r+\frac 1{\ve^2} r
    + F(\ve,t,x)=0,
  \]
we have that $r$ presents of order $O(\ve^a), a\ge 2$ if initial data is small enough and the nonlinearity $F(\ve,t,x)$ involves the terms like
\begin{center}
   $O(\ve^a) $; \quad 
 oscillation $\ve^{a-2}\ei{ m}\cdot O(1), m\neq 1$; \quad
bootstrap structure like $O(1)\cdot r+O(1)\cdot r^2+r^3$; 
\end{center}
or their linear combinations,
See Lemmas \ref{lem:estimate of r with small perturbation}--\ref{lem:estimate of r with bootstrap} for more details.
  Guided by this principle,
  we first observe that $F_2(\Phi_1, \Phi_2, r_1)$ formally satisfies the above forms for $a>2$.
  Therefore, we can expect $r_1$ to be of a higher order than $O(\ve^2)$ if $F_1(\Phi_1, \Phi_2, v)$ satisfies formally the principle with $a > 2$ given $\Phi_1, \Phi_2\sim O(1)$.
  To clarify the order of each term in  $F_1(\Phi_1, \Phi_2, v)$,
  we employ the idea of the modulated Fourier expansion and set reasonably
     \begin{align*}
            \Phi_2 = \ei{} \eta_1 +\ei{3} \eta_2+c.c..
         \end{align*}
 The choice of the expansion for $\Phi_2$ is mainly due to the presence of the factors $\ei{\pm }$ and $\ei{\pm 3}$ in \eqref{eq:r}.
Substituting this formula of $\Phi_2$ into the equation of $r_1$ and 
     we get 
     \begin{align*}
    \ve^2 \p_{tt} r_1 -\De r_1 +\frac{1}{\ve^2} r_1
    +F_\ve
    =0,
\end{align*}
where $F_\ve$ is given by 
\begin{align*}
	F_\ve 
	&=\ei{} \ve^2 \big[
			2i\p_t \eta_1-\De \eta_1+3(2\eta_1|v|^2
			+v^2\ba \eta_1+ \ba v^2 \eta_2)
			+\p_{tt}v
		\big]
        +\ei{3}\big(v^3 -8\eta_2\big)
	\\
	&\quad
	+\ve^2 \ei{m} O(1)
    +\ve^4 O(1) 
    +F_2(\Phi_1, \Phi_2, r_1)
	+c.c..
\end{align*}
     Let the coefficients of $\ve^2 \ei{\pm } $ and $\ei{\pm 3}$ be zero, which implies \eqref{eq:f} and \eqref{Phi_2},
     then $F_\ve $ formally satisfies the above principle with $a >2$. Moreover, $r_1$ can be identified of order $O(\ve^4)$ by further analysis.   
  We note that the profile $w$ may be not uniquely defined, see Remark \ref{rem 1} below.
  However, we believe that
  the high-order oscillator term $\frac 18  \ei{3} v^3 +c.c. $
  arises inevitably in the $O(\ve^2) $ leading term $\Phi_2$, which verifies the observation in numerical experiments \cite{schratz2020} from the theoretical perspective.

  \subsection{Main theorem}
  In this paper, we focus on the high-order asymptotic expansion \eqref{expansion of u} and aim to show the convergence rate of $r_1$ is of order $O(\ve^4)$.
  This presents a higher convergence rate for the error $r_1$, compared to the $r$-estimate appeared in the previous result \cite{lei}.
At the same time, our result indicates the validity of the high-order asymptotic expansion \eqref{expansion of u}. 
  Moreover, as shown in the subsequent theorem, 
  the convergence rate of $r_1$ depends on time and the regularity of initial data, which is similar to the first-order convergence result \cite{lei}.
  
  We state the main theorem as follows. 
  
  \begin{thm}[Convergence rate]
      \label{main theorem}
      Let $d=2,3$ and $u_0,u_1 \in H^{\al}_x(\R^d)$ with $\al\in [4, 8]$.
      Assume that $\Phi_1$ and $ \Phi_2$ are given in Definition \ref{defn;main},
      then we have for all $ t \ge 0 $,
      \begin{align*}
          \nl{u^\ve -\Phi_1 -\ve^2 \Phi_2}{2}
          \le C\big(\ve^4+\left(\ve^2t\right)^\frac{\al}{4}\big).
      \end{align*}
      where the constant $C$ depends on $\norm{u_0}_{H_x^{\al}}$,
      $\norm{u_1}_{H_x^{\al}}$.
  \end{thm}
  
  \begin{remark}
      \label{rem 1}
      We remark that the choice of $w$ is not unique.
      Indeed, we can further decompose $w$ by assuming that
      \[ w(t,x):= a\De v +b|v|^2v +g(t,x),  \]
      where $a$ and $b$ can be arbitrary constants, and
      $g$ satisfies the following Schr\"odinger-type equation:
      \eqn{
          &2i\p_t g -\De g+3 \left( v^2\ba g+ 2|v|^2g  \right)+F(v)=0, \\
          &g(0) =w_0-a\De v_0-b|v_0|^2 v_0,
      }
      where 
      \begin{align*}
          F(v)&=-\frac 14 \De^2 v
          +\left(-3a-b+\frac 34 \right) \De \left(|v|^2v\right)
          +\left(6a+2b+\frac 32 \right)|v|^2 \De v
          +\left(3a-b-\frac 34 \right)v^2\De \ba v 
          \\
          &\quad
          +\left(6b-\frac {15}8\right)|v|^4v. 
      \end{align*}
      As it presents, there is always a $\De ^2 v $-term included in $F(v)$, which results in time-dependent bounds for both $w$ and $ g$. The remaining terms can be addressed similarly in our analysis.
      Consequently, the analysis of $w$ and $g$ is essentially identical regardless of the constants $a$ and $b$.
  \end{remark}
  \begin{remark}
      The value of $\al\ge 4 $ can be seen as a continuous extension to the first asymptotic expansion in \cite{lei} which requires $H^\al$-data, $\al\in [1,4]$.
      We also note for $\al\ge 8$,
      the upper bound $O(\ve^4+\ve^4t^2)$ is optimal,
      see Theorem \ref{optimal rate} below.  
  \end{remark}

  We show the optimality of convergence rate in Theorem \ref{main theorem} by constructing the following counter-examples.
  \begin{thm}[Counter-examples]
      \label{optimal rate}
      Let $d=2,3$, 
      and assume $\Phi_1, \Phi_2$ are given in Definition \ref{defn;main}.
      Then there exists an initial data in $ \mathscr S (\R^d)$  such that the solution $u^\ve$ of \eqref{eq:u} satisfies
      \begin{align*}
          \nl{u^\ve -  \Phi_1-\ve^2\Phi_2 }{2}
          \ge C_1  \ve^4t^2,
          \quad
          1 \le t \le \ve^{-2}\de_0,
      \end{align*}
      and an initial data in $ H^\al _x(\R^d), \al \in [4, 8]$ such that the related solution satisfies
      \begin{align*}
          \nl{u^\ve -  \Phi_1-\ve^2\Phi_2 }{2}
          \ge C_2  \left(\ve^2 t\right)^\frac \al 4
          \big|
          \ln (\ve^2t)
          \big|^{-1},
          \quad
          1 \le t \le \ve^{-2}\de_0.
      \end{align*} 
      Here, the constants $\de_0$ that is suitably small, and $C_1, C_2$ are independent of $\ve$ and $t$.
  \end{thm}
  \begin{remark}
      Here $t\ge 1$ guarantees that $(\ve^2t)^\frac \al 4 $ is the main dominant term in the convergence rate of $r_1$ given in Theorem \ref{main theorem}.
      In addition, $t \le \ve^{-2}\de_0$ is necessary due to the bootstrap argument used in our analysis.
  \end{remark}
  
  Inspired by \cite{lei}, 
  we give counter-examples based on the strategies for the different phase characteristics. 
  Extracting continually the dominant term from the Duhamel formula of $r_1$,
  we encounter the dominant term that is included in the following forms:
  \begin{align*}
      \int_0^t e^{\pm is(\jd{}+1-\frac \De 2)} F(h_0) ds 
      \quad
      \mbox{ and } 
      \quad
      \int_0^t e^{\pm is(-\jd{}+1-\frac \De 2)} F(h_0) ds.
  \end{align*}
  Then the strategies can be stated as follows.
  For the first integral above, 
  we observe that the phase is away from zero after applying the Fourier transform, 
  and then employing integration by parts 
  gives a higher order in $\ve$. Therefore, this term can be absorbed into the remainder term. 
  Moreover, the phase in the second integral may be closed to zero when the frequency is near zero after applying the Fourier transform. Then, we approximate the phase utilizing some linear operator.
  More specifically, we have
  \begin{align*}
      e^{it\left(-\jd{}+1-\frac \De 2\right)}
      =1+\frac18 it\De^2 
      +O(t\De^3),
  \end{align*}
  which can be deduced from Taylor's expansion and Mihlin-H\"ormander Multiplier Theorem.
  This formula implies that the dominant term is included in
  the part of the action of the operator $1+\frac 18 it\De ^2$.
  The term including $O(t\De^3)$ will be controlled by $O(t\ve^6)$,
  and can be regarded as the remainder term.
  With the help of both strategies, 
  we can find the dominant term, and analyze it further to give counter-examples.
  
  \subsection{Organization}
  This paper is structured as follows. 
  In Section \ref{preliminary}, we present some basic notation and useful lemmas. 
  In Section \ref{the high-order asymptotic expansion},
  we give some crucial lemmas as principles to deduce the $O(\ve^2)$ leading term in the high-order asymptotic expansion of $u^\ve$.
  In Section \ref{the key estimates}, we establish the space-time estimate for $w$ to obtain the space-time estimate for the nonlinear part of $\phi$, applying the Strichartz estimates and bootstrap argument. 
  In Section \ref{Non-relativistic limit}, 
  with key estimates established in previous sections, 
  we demonstrate Theorem \ref{main theorem} in the regular case, and further
  extend it to the non-regular case.
  The proof relies on the high-low frequency decomposition
  and bootstrap argument.
  In Section \ref{Optimal convergence rate}, 
  we construct counter-examples 
  to show Theorem \ref{optimal rate}.

\section{Preliminary}
\label{preliminary}
\subsection{Notations}
\label{notations}
Denote $\langle \cdot \rangle = \left( 1+|\cdot|^2 \right) ^{\frac12}$.
We use $ A\lesssim B $ or $  A = O(B) $ to denote $ A \le CB $ 
and use $ A \ges B $ to denote $A \ge CB$ where $ C > 0 $ is a constant.
Moreover, we denote $A\sim B $ if $A\les B $ and $B\les A $.

For a function $ g(x) $, we denote the Lebesgue's and Sobolev's space norms for $g(x)$:
\begin{align*}
    \|g\|_{L_x^q}&:=\|g\|_{L^q_x(\R^d)}
    : =\left(\int_{\mathbb{R}^d} |g(x)|^q dx \right)^{1/q}, 
        \\
    \|g\|_{\dot H^s_x}&:=\||\nabla|^s g\|_{L_x^2}, 
    \quad
  \|g\|_{H^s_x} := \left( \|g\|^2_{L_x^2}
            +\|g\|_{\dot H^s_x}^2 \right) ^{\frac{1}{2}} . 
\end{align*}
We denote the space-time norm for $g(t,x)$:
\begin{align*}
    \norm{g}_{L_t^qL_x^p(I)} := \left( \int_{I}^{} \left( \int_{\R^d}^{} |{g(t,x)}|^pdx\right) ^{\frac qp} dt \right) ^{\frac1q} 
\end{align*}
and denote $L^q_{tx}(I)$ for $L^q_tL^q_x(I)$.

Let $\phi(\xi)$ be a radial smooth bump function supported in the ball
$\{\xi \in \R^d: |\xi |\le \frac{11}{10} \}$ such that $\phi(\xi)=1$ if $|\xi |\le 1$, and $0<\phi(\xi)<1$ if $1\le |\xi |\le \frac{11}{10}$. 
Then for $N\in 2^\N$, we denote frequency cut-off operators
\begin{align*}
   P_{\le N}g :&= \F^{-1} \left(
                           \phi\left(\frac \xi N \right)\F g     
                    \right),\\
   P_{>N}g:&=g-P_{\le  N} g,               
\end{align*} 
where $\F$ denotes the Fourier transform:
\[ \F g(\xi) :=\int_{\R^d}^{} e^{-ix\xi} g(x)dx, \]
and the inverse Fourier transform $\F^{-1}$ is given by
$\F^{-1}g(\xi) :=(2\pi)^{-d}\F g(-\xi)$.

For convenience, we give some special notation which will be used in the following sections.
Denote
$$s_c:=\frac{d}{2}-1,$$ 

We also denote the time-space norm
\[     \sc g:=\norm{|\na|^{s_c}g}_{L_{tx}^{\frac{2(d+2)}{d}}}, \]
and the following workspace norms: 
\begin{align*}
    \xsc g&:= \norm{g}_{L_t^{\oo}\dot H^{s_c}_x}
               +\sc g,\\
    \ysc g&:=\norm{\jd{\frac12}g}_{L_t^{\oo}\dot H^{s_c}_x}
    +\sc g, \\
    \zsc g&:=\ntx{|\na|^{s_c}g}{\frac{2(d+2)}{3d}}{\frac{2d(d+2)}{8+d^2}},
\end{align*}
where all the spacetime norms are taken over $(t,x)\in I\times \R^d$ with some interval $I\subset \R$.

\subsection{Basic lemmas}  
We give some basic and useful lemmas.

First, we recall the fractional Leibniz rule first proved by Christ and Weinstein \cite{christ1991}.
\begin{lem}[Fractional Leibniz rule]
    \label{chain rules}
    Let $s> 0$, $1< p,p_2,p_4\le \oo, 1< p_1, p_3 <\oo$ satisfy
    $\frac1p=\frac1{p_1} +\frac 1{p_2}
    =\frac1{p_3}+\frac1{p_4}$,
    there holds that for $f, g \in \mathscr S\left(\R^d\right)$,
    \[ 
    \norm{|\nabla|^s (fg)}_{L_x^p}
    \le C\left(\nl{|\nabla|^s f}{p_1}\nl{g}{p_2}
    + \nl{|\nabla|^s g}{p_3}\nl{f}{p_4}\right),
    \]
    where the implicit constant $C$ depends on $s, p,p_1,p_2,p_3$, and $p_4.$
\end{lem}

Recall the definition of $\dot H_x^\ga$-admissible. 
\begin{definition}[$\dot H_x^\ga$-admissible]
    Let $I \subset \R$ and $ 0\le \ga\le 1$. The pair $(q, r)\in \R^2$ is called 
    $\dot H_x^\ga $-admissible, if  $(q,r)$ satisfies $2\le q, r\le \oo$, $(q,r,d)\neq (2, \oo,2)$,
    and 
    \begin{align*}
        \frac2q+\frac dr=\frac d 2-\ga .
    \end{align*}
    In particular, $(q,r)$ is $L_x^2$-admissible if $\ga=0.$
\end{definition}    
Now we can state the Strichartz estimates as follows, see \cite{keel1998} for instance.
\begin{lem}[Strichartz estimates for NLS] 
    \label{strichartz estimate for NLS}
    Let $I \subset \R$. Suppose that $(q,r )$  and $(\ti q, \ti r)$ are 
    $L_x^2$-admissible. Then 
    \[ \norm{e^{it\De }\varphi}_{L^q_tL^r_x(I)} \les  
    \norm{\varphi}_{L^2_x}, \]
    and
    \[ \norm{\int_{0}^{t}e^{i(t-s)\De }F(s)ds}_{{L^q_tL^r_x(\R)}}
    \les 
    \norm{F}_{L^{\ti q'}_tL^{\ti r '}_x(\R)}.
    \]
\end{lem}

\begin{lem}[Strichartz estimates for NLKG]
   \label{strichartz estimate}
    Let $I \subset \R$. Suppose that $(q,r )$  and $(\ti q, \ti r)$ are 
    $L_x^2$-admissible. Then 
    \[ \norm{e^{it\jd{}}\varphi}_{L^q_tL^r_x(I)} \les  
    \norm{\jd{\frac{d+2}{2}\left(
            \frac12-\frac 1r\right)}\varphi}_{L^2_x}, \]
        and
     \[ \norm{\int_{0}^{t}e^{i(t-s)\jd{}}F(s)ds}_{{L^q_tL^r_x(\R)}}
        \les 
        \norm{\jd{\frac{d+2}{2}\left(
                    1-\frac 1r - \frac 1{\ti r}
                \right)}F}_{L^{\ti q'}_tL^{\ti r'}_x(\R)}.
      \]
\end{lem}

\begin{remark}\label{rem:strichartz estimates}
    We give a comparison in the framework of Strichartz estimates between cubic NLS and NLKG. 
    More specifically, 
    assuming the following Duhamel formulas:
    \[ \mbox{\emph{NLS}}: \qquad
    \varphi_1(t) =e^{it\De }\varphi_1(0)
    +\int_0^t e^{i(t-s)\De} \left( |\varphi_1|^2\varphi_1\right) (s)ds,  \] 
    and 
    \[ \mbox{\emph{NLKG}}: \qquad
    \varphi_2(t) =e^{it\jd{}  }\varphi_2(0)
    +\int_0^t e^{i(t-s)\jd{}}\jd{-1}\left( |\varphi_2|^2\varphi_2\right)(s)ds.  \]
    Then by Lemmas \ref{strichartz estimate for NLS} and \ref{strichartz estimate}, we have
    \[ \xsc{\varphi_1} 
            \les \nh{\varphi_1(0)}{s_c} 
                +\zsc{ 						    
                    |\varphi_1|^2\varphi_1}
            \les \nh{\varphi_1(0)}{s_c}
                + \sc{\varphi_1}^3,\]
    and
    \[ \ysc{\varphi_2} 
            \les \nh{\jd{\frac 12}\varphi_2(0)}{s_c}
                + \zsc{|\varphi_2|^2\varphi_2}
            \les \nh{\jd{\frac 12}\varphi_2(0)}{s_c}
                +\sc{\varphi_2}^3.
            \]                     
\end{remark} 

We review the scattering results for the defocusing cubic NLS, serving as the foundational facts in this paper. For $d=2$, the scattering was initially established by Killip, Tao, and Visan \cite{killip2009} in the radial case.  Subsequently, Dodson \cite{dodson2016} extended the result to the nonradial case. 
As for the $d=3$ case, the scattering was proved by Lin-Strauss \cite{lin1978}, and Ginibre-Velo \cite{ginibre1992}. 
We summarize the scattering results for 2D and 3D cubic NLS as follows:
    
\begin{lem}[Scattering theory for cubic NLS]
    \label{scattering theory}
    Let $d=2,3$, and $ v_0\in H^s\cap H^{d-2}$, $s\ge0$.
    Assume that $v$ satisfies NLS \eqref{eq:v} and $(q,r )$ is $L^2$-admissible. Then 
    \[ \ntx{|\nabla|^sv}{q}{r} 
    \le C\left(\norm{v_0}_{H_x^{d-2}}\right) \nh{v_0}{s}. 
        \]
\end{lem}
\begin{remark}\label{rem:chain rules}
    It follows from Lemmas \ref{chain rules}, \ref{strichartz estimate for NLS}, and
    \ref{strichartz estimate} that  
    for $ -s_c \le \ga\le  0$ and $ d=2,3,$
    \[        \zsc{ 
        \naabs{\ga} \left( \varphi_1 \varphi_2\varphi_3 \right)}
    \les \sc{ \naabs{\ga} \varphi_1} \sc{\varphi_2} \sc{\varphi_3}. \]   
\end{remark}

\section{The high-order asymptotic expansion}
\label{the high-order asymptotic expansion}

In this section, we investigate the high-order asymptotic expansion of $u^\ve(t)$. To achieve this, we need to isolate the first-order error terms and identify the $O(\ve^2)$ leading term. 
To this end, we first give a general scheme to show which type of nonlinearity would lead to $O(\ve^a)$-error.

\subsection{Nonlinearities which implies $O(\ve^a)$ error}

 Let $d=2,3.$ Suppose that $r(t,x)$ satisfies the following general equation:
\eq{
    \label{the equation of r with initial data}
    &\ve^2\p_{tt} r-\De r+\frac 1{\ve^2} r
    + F(\ve,t,x)=0,
    \\
    &r(0,x)=r_0(x),
\quad
\p_t r(0,x) = \wt r_0(x),  
}
with initial data satisfying
\begin{align}
    \label{initial data}
    \norm{r_0}_{H_x^{\frac 12}}\les \ve^{a}
    \quad
    \mbox{ and }
    \quad
    \nl{\wt r_0}{2} \les \ve^{a-2},
    \qquad a\ge 2. 
\end{align}
We shall show that there exist three types of nonlinearity $F(\ve, t, x)$, which imply that the $L^2$-norm of $r$ persists the same error estimate as the initial data $r_0$ given in \eqref{initial data}:
\begin{itemize}
\item 
$F(\ve, t, x)$ is of the type $\ve^{a}\cdot O(1)$, which means that it possesses the same smallness as the initial data.
\item 
$F(\ve, t, x)$ is of the type $\ve^{a-2}e^{ i\frac{m}{\ve^2}t}\cdot O(1)$, $m\ne 1$, $m\in \R$ with oscillation.
\item 
$F(\ve, t, x)$ involves a bootstrap structure.
\end{itemize}
These facts can be stated precisely as Lemmas \ref{lem:estimate of r with small perturbation}--\ref{lem:estimate of r with bootstrap}.
\begin{lem}[Small perturbation]
    \label{lem:estimate of r with small perturbation}
    Let $d=2,3$ and $r(t,x): \R \times \R^d \to \R$ is a real-valued function
    and verifies equation \eqref{the equation of r with initial data}  with initial data conditions \eqref{initial data}. 
    Assume that the nonlinearity $F(\ve,t,x )$ satisfies
     $$F(\ve,t,x ) =\ve^{a} G(\ve, t,x), \quad a\ge 2,$$ 
     and $G(\ve, t,x)$ satisfies 
    \begin{align*}
        \zo{G}\le C_1.
    \end{align*}   
    Then we have 
    \begin{align*}
        \ntx{r}{\oo}{2}\le C_2 \ve^{a}.
    \end{align*}
    Here constants $C_1, C_2$ are independent of $\ve $ and $t$.
\end{lem}
\begin{proof}
     First, we introduce the useful scaling transformation $S_\ve $ defined by
    \[ S_\ve g(t,x):= \ve g(\ve^2 t, \ve x). \] 
Note that the $X_{s_c}$ norm is invariant under the scaling transformation $S_\ve$, that is 
\[ \xsc{S_\ve g} =\xsc{g}.  \] 
    We set 
    \begin{align*}
        R= S_\ve r, 
    \end{align*}
    and denote
    \begin{align*}
        \phi = (\p_t +i \jd{})\jd{-1}{R}, 
    \end{align*}
    then we have 
    \eq{ \label{3.1}
        &{R}=\im \phi,\\
        &\p_t {R} =\re \jd{}\phi.}
    The equation \eqref{the equation of r with initial data} of $r$ can be rewritten as follows: 
    \begin{align}
        (\p_t -i \jd{})\phi =-\jd{-1}\ve^2 S_\ve F(\ve,t,x ),
        \label{}
    \end{align}
    with initial data 
    \begin{equation}
        \label{3.3}
        \phi(0):={\phi_0}=\jd{-1}\p_t{R}(0)+i{R}(0)=\ve^2\jd{-1}S_\ve (\p_t {r}(0))
        +iS_\ve r(0).
    \end{equation}
    The Duhamel formula reads
    \begin{align*}
        \phi(t)=e^{it\jd{}}{\phi_0} 
        -\int_{0}^{t}e^{i(t-s)\jd{}}\jd{-1}\ve^2 S_\ve F(\ve,s,x ) ds.
    \end{align*}
    By Lemma \ref{strichartz estimate}, Remark \ref{rem:strichartz estimates}, and noting that
    \begin{align*}
        \zo{S_\ve F} = \ve^{-2-s_c} \zo{F},
    \end{align*} 
    we have 
    \begin{align}
        \label{3.4}
        \yo{\phi}
        &\les \yo{e^{it\jd{}}{\phi_0} }
        + \ve^{2} \zo{S_\ve F(\ve,t,x )}
        \notag \\
        &\les \yo{e^{it\jd{}}{\phi_0} }
        + \ve^{-s_c} \zo{F(\ve,t,x )} .  
    \end{align}
    It follows from \eqref{initial data}, \eqref{3.3}, and Lemma \ref{strichartz estimate} that
    \begin{align}
        \label{3.5}
        \yo{e^{it\jd{}}{\phi_0} }
        &\les \nl{\jd{\frac 12}\phi_0}{2}
        \notag\\
        &\les \ve^{2-s_c}  
        \nl{\tilde r_0}{2}
        +\ve^{-s_c}\nl{\jd{\frac 12}r_0}{2}
        \notag\\
        &\les \ve^{a-s_c}.
    \end{align}
    For $F(t,x )=\ve^{a} G(\ve,t,x )$, we obtain 
    \begin{align*}
        \zo{F} &\les \ve^{a}\zo{G}
        \les \ve^{a}. 
    \end{align*}
    In view of the above estimate, \eqref{3.4}, and \eqref{3.5},
    we have 
    \begin{align*}
        \yo{\phi} \les \ve^{a-s_c}. 
    \end{align*}
    This inequality and \eqref{3.1} give that 
    \begin{align*}
        \ntx{r}{\oo}{2}
        =\ve^{s_c } \ntx{R}{\oo}{2}
        \les \ve^{s_c }\yo{\phi}
        \les 
        \ve^{a}, 
    \end{align*}
    which finishes the proof of this lemma.
\end{proof}
\begin{lem}[Oscillation]
    \label{lem:estimate of r with high oscillation}
    Under the assumptions of Lemma \ref{lem:estimate of r with small perturbation} with $F(\ve, t,x )$ replaced by 
    $$F(\ve,t,x )=\ve^{a-2}e^{\frac { imt}{\ve^2}}G(t,x ), \quad
     m \neq 1, m\in \R, a\ge 2 $$   
     where $G(t,x)$ is independent of $\ve$ and satisfies 
    \begin{align*}
        \xo{G}, \zo{\De G}, \zo{\p_t G}
        \le C_1.
    \end{align*}
        Then we have 
    \begin{align*}
        \ntx{r}{\oo}{2}\le C_2\ve^{a}. 
    \end{align*}
    Here constants $C_1, C_2$ are independent of $\ve $ and $t$.
\end{lem}   
\begin{proof}
    We perform the same procedure as shown in the proof of Lemma \ref{lem:estimate of r with small perturbation} and obtain that 
    \begin{align*}
        \yo{\phi}
        &\les \yo{e^{it\jd{}}{\phi_0} }
        + \ve^{a} \yo{e^{it\jd{}}\int_{0}^{t}e^{is(  m-\jd{})}\jd{-1} S_\ve G(s,x ) ds}.
    \end{align*}
    Invoking \eqref{3.5} and arguing like Lemma \ref{lem:estimate of r with small perturbation},
    it remains to show the $Y_0$-norm bound is $O(\ve^{s_c})$ for the following integral:
    \begin{align*}
         e^{it\jd{}}\int_{0}^{t}e^{is( m-\jd{})}\jd{-1} S_\ve G(s,x ) ds.
    \end{align*}  
    We only consider the case $m>0$ since the case $m\le0$ is similar and easier.
    For fixed $m$, we divide the above integral into the $P_{\le k}$ part and the $P_{>k}$ part,
    \begin{align*}
        (P_{\le k}+P_{>k} ) e^{it\jd{}}\int_{0}^{t}e^{is( m-\jd{})}\jd{-1} S_\ve G(s,x ) ds,
    \end{align*}
    where $k $ is some positive constant satisfying
    \begin{align*}
        \F \Big({P_{\le k}\big(
            {m-\jd{}}} \big)
        \Big)\neq 0. 
    \end{align*}
   Then for low frequency $P_{\le k}$ part, integration by parts shows that 
    \begin{align*}
        &P_{\le k}e^{it\jd{}}\int_{0}^{t}e^{is( m-\jd{})}\jd{-1} S_\ve G(s,x ) ds
        \\
        &\quad
        =\frac{\jd{-1}}{i(m-\jd{})}P_{\le k}
        \Big[
        e^{imt}S_\ve G(t,x )-e^{it\jd{}}S_\ve G(0,x)
        -\ve^2e^{it\jd{}}\int_{0}^{t}e^{is(m-\jd{})}S_\ve
        (\p_sG(s,x))ds
        \Big].
    \end{align*}
    This implies that 
    \begin{align*}
        &\yo{P_{\le k}e^{it\jd{}}\int_{0}^{t}e^{is( m-\jd{})}\jd{-1} S_\ve G(s,x ) ds}
        \\
        &\quad
        \les \yo{\jd{-1}S_\ve G}+\yo{\jd{-1}e^{it\jd{}}S_\ve G(0, x)}
        +\ve^2\zo{S_\ve (\p_t G)}
        \\
        &\quad
        \les 
        \xo{S_\ve G}
        +\ve^2\zo{S_\ve (\p_t G)}
        \\
        &\quad
        \les \ve^{-s_c}(\xo{G}+\zo{\p_tG})
        \\
        &\quad
        \les \ve^{-s_c}.
    \end{align*}
     For the high-frequency case, we have 
    \begin{align*}
        \yo{P_{>k} e^{it\jd{}}\int_{0}^{t}e^{is( m-\jd{})}\jd{-1} S_\ve G(s,x ) ds}
        &\les \zo{P_{>k}S_\ve G(s,x ) }\\
        &\les \zo{\De S_\ve G(s,x ) }\\
        &\les \ve^{-s_c} \zo{\De G}\\
        &\les \ve^{-s_c}.
    \end{align*}
    Combining estimates of both cases for $P_{\le k}$ and $ P_{>k}$, we establish the desired conclusion.
\end{proof} 
\begin{lem}[Bootstrap structure]
    \label{lem:estimate of r with bootstrap}
    Under the assumptions of Lemma \ref{lem:estimate of r with small perturbation} but with $F(\ve, t,x )$ replaced by 
    $$F(\ve,t,x)= (g_1(\ve, t,x ))^2r+g_2(\ve, t,x)r^2+r^3,$$   
     where $g_1(\ve, t,x )$ and $ g_2(\ve, t,x )$ satisfy 
    \begin{align*}
        \norm{g_1(\ve, t,x )}_{S_c(\R)}+
        \norm{g_2(\ve, t,x )}_{S_c(\R)} \le C_1. 
    \end{align*}
    Then we have 
    \begin{align*}
        \ntx{r}{\oo}{2}\le C_2 \ve^a. 
    \end{align*}
    Here the constants $C_1$ and $C_2$ are independent of $\ve $ and $t$.
\end{lem} 

\begin{proof}
   Arguing similarly as the proof of Lemma \ref{lem:estimate of r with small perturbation}, we obtain       
        \begin{align*}
           \ysc{R}
           \les \ysc{\phi}
           &\les \nh{\jd{\frac12}\phi_0}{s_c} 
           +\zsc{F}
           \\
           &\les \nh{\jd{\frac12}\phi_0}{s_c} +\sc{ S_\ve g_1}^2 \sc{R}
           +\sc{R}^2
           +\sc{R}^3.
        \end{align*}
        Since $\norm{g_1}_{S_c(\R)}< \oo$, then there exists a constant $K$ and a sequence of time intervals 
        satisfying 
        \begin{align*}
            \bigcup^K_{k=1}J_k=\R^+,
        \end{align*}
        with
        \begin{align*}
            J_0=[0,t_1], \quad
            J_k=(t_k,t_{k+1}], 
            \quad
            k=1,\dots,K-1, \quad
            J_K=[t_K,\oo),
        \end{align*}
        such that for $I_k=\ve^2 J_k \cap [0,t]$,
        \begin{align*}
            C\norm{S_\ve g_1}^2_{S_{s_c}(I_k)}\leq\frac{1}{2}.
        \end{align*} 
        Thus we have for any $I_k, 0\le k \le K$,
        \begin{align*}
           \norm{R}_{Y_{s_c}(I_k)}
            \les \norm{\jd{\frac12}\phi(t_{k})}_{\dot H_x^{s_c}}
            +\norm{R}_{S_{s_c}(I_k)}^2
            +\norm{R}_{S_{s_c}(I_k)}^3.
        \end{align*}
        Iterating with respect to $k$ from $0$ to $K$ and noting 
        \begin{align*}
           \nh{\jd{\frac12}\phi_0}{s_c}
           \les \ve^a, 
        \end{align*}
         a standard bootstrap argument shows that for any $I_k, 0\le k\le K$,
        \begin{align*}
            \ysc{R}
            \les \ve^a.
        \end{align*}
        On the other hand, it follows that  
        \begin{align*}
           \yo{R}\les \yo{\phi}
           &\les 
           \ve^{a-s_c} + \zo{F} \\
           &\les \ve^{a-s_c} 
            +\sc{S_\ve g_1}^2 \so{R}
            +\sc{S_\ve g_2}\sc{R}\so{R}
            +\sc{R}^2\so{R} 
            \\
          &\les 
            \ve^{a-s_c} 
            +\sc{S_\ve g_1}^2 \so{R}
            +\ve^a\so{R}
            +\ve^{2a}\so{R}.
        \end{align*}
        Arguing analogously as the preceding iteration and bootstrap procedure, we see that
        \begin{align*}
           \yo{R}\les \ve^{a-s_c}, 
        \end{align*}
        which implies the desired result.
\end{proof}
We make the following remarks regarding Lemmas \ref{lem:estimate of r with small perturbation}--\ref{lem:estimate of r with bootstrap}.
\begin{remark}
The same result also holds when the nonlinearity $F(\ve,t, x)$ is replaced by the linear combination of the three types mentioned above, by combining the proofs of these lemmas. 
\end{remark}

\begin{remark}
We also note that the $Z_0$ bound of $G(\ve, t, x) $ in Lemma \ref{lem:estimate of r with small perturbation} can depend on time, which finally leads to the time-growth bound for the corresponding solution, see \cite{lei}.
It is worth noting that this time-growth bound will influence the bootstrap argument in Lemma \ref{lem:estimate of r with bootstrap} when time grows.
Consequently, we may only get the $\ve$ and time-dependent bound of the corresponding solution for a short time.
To gain the same bound for the global time,  
we have to seek the help of the conservation law, 
see Section \ref{Non-relativistic limit} for example.
\end{remark}

\subsection{Derivation for the high-order asymptotic expansion}
In this subsection, we aim to derive the high-order asymptotic expansion of $u^\varepsilon$ guided by Lemmas \ref{lem:estimate of r with small perturbation}--\ref{lem:estimate of r with bootstrap}.

$\bullet$ \textit{Renormalization of the error equation.} According to the first-order asymptotic expansion \eqref{first order expansion of u} and the result from \cite{lei} that the first-order remainder term $r$ is of order $O(\ve^2)$, we need to expand the first-order remainder term $r$ further to a sum consisting of an $O(\ve^2)$ leading term and a higher-order remainder term.
Recall the equation of $r$ in \eqref{eq:r},
\eqn{
    &\ve^2 \p_{tt} r -\De r +\frac{1}{\ve^2} r
     + F_1(v)+F_2(v) + F_3(v, r)=0, \\
    &r(0)=0, \quad  \p_t r(0) = -(\p_t v(0)+\p_t \ba v(0)),	 
}
where $v$ and $\Phi_1$ are defined in \eqref{eq:v} and \eqref{Phi_1}, and
\begin{align*}
    F_1(v): = \ve^2 
    \ei{ }\p_{tt}v+ c.c.,
    \quad
    F_2(v):= \ei{3}v^3+c.c.,
    \quad
    F_3(v,r):= 3{\Phi_1}^2 r+3{\Phi_1}r^2+r^3.
\end{align*}
To begin with, we undertake the decomposition of $r$ in the following manner:
\begin{equation*}
    r:= \ve^2 {\Phi_2} +r_1, 
\end{equation*} 
where  we expect that $\ve^2 \Phi_2$ with $\Phi_2\sim O(1)$ is the $O(\ve^2)$ leading term, and $r_1$ is the corresponding high-order remainder term.
Inserting this formula into \eqref{eq:r}, we have 
\begin{align}
   &\ve^2 \p_{tt} r_1 -\De r_1 +\frac{1}{\ve^2} r_1
   +\ve^4 \p_{tt}\Phi_2 -\ve^2 \De \Phi_2 +\Phi_2
   + 3\ve^2 \Phi_1^2\Phi_2
   +F_1(v)
   +F_2(v)
   \notag\\
  &\quad
    +3\ve^4 {\Phi_1}\Phi_2^2 +\ve^6 \Phi_2^3
    +
    3\left({\Phi_1}+\ve^2 {\Phi_2}\right)^2 {r_1}
    + 3\left({\Phi_1}+\ve^2 {\Phi_2}\right) r_1^2
    +r_1^3
  =0.
  \label{second order equation}
\end{align}  

$\bullet$ \textit{Isolating of the $O(\ve^2)$-terms in \eqref{second order equation}.} 
To derive the expression for $ \Phi_2 $, we need to identify which nonlinearities in \eqref{second order equation} are precisely of order $ O(\varepsilon^2) $. 
We then cancel these nonlinearities using terms generated by $ \Phi_2 $ to ensure that the remaining nonlinearities are of a higher order. Consequently, we expect the $ L^2 $ estimate of $ r_1 $ to be of higher order than $ O(\varepsilon^2) $.

We claim that presenting exactly order $ O(\ve^2)$ are terms $3\ve^2 \Phi_1^2\Phi_2, F_1(v), F_2(v)$ in \eqref{second order equation}.   
First, now that $\Phi_1, \Phi_2 \sim O(1)$ and by Lemma \ref{lem:estimate of r with small perturbation}, we neglect the dependence of time and identify that formally
\begin{align*}
    3\ve^2 \Phi_1^2\Phi_2 \mbox{ and } F_1(v)\sim O(\ve^2).    
\end{align*}
By virtue of Lemma \ref{lem:estimate of r with high oscillation},
we find that the effect of the factor $\ei{m}, m\neq 1, m\in \R$ in nonlinearity for equation \eqref{eq:r} is equivalent to a factor $\ve^2$. In this sense, we have implicitly 
\[ F_2(v)\sim O(\ve^2). \]
Here we note that $\ve^4 \p_{tt}\Phi_2 -\ve^2 \De \Phi_2 +\Phi_2$ is at least of order $O(\ve^2)$.
Secondly, for those nonlinearities in the second line of \eqref{second order equation},  we apply Lemmas \ref{lem:estimate of r with small perturbation} and \ref{lem:estimate of r with bootstrap},
and see that formally
 \begin{align*}
   3\ve^4 {\Phi_1}\Phi_2^2 \sim O(\ve^4) ,
\quad
\ve^6 \Phi_2^3  \sim O(\ve^6),
\end{align*}
and
\begin{align*}
   3\left({\Phi_1}+\ve^2 {\Phi_2}\right)^2 {r_1}
+ 3\left({\Phi_1}+\ve^2 {\Phi_2}\right) r_1^2
+r_1^3  
\end{align*}
involves a bootstrap structure which can be regarded as a higher order term than $O(\ve^2)$.

$\bullet$ \textit{Further renormalization by the modulated Fourier expansion.} Next, 
our task is to cancel the $O(\ve^2)$-term using $\Phi_2$.
To get a more precise expression presenting of order $O(\ve^2)$ from the nonlinearity of \eqref{second order equation},
we employ the idea of the modulated Fourier expansion and further expand $\Phi_2$ into different $\ve$-dependent frequencies,
\begin{align*}
    \Phi_2=\ei{}\eta_1+\ei{3}\eta_2+c.c.,
\end{align*}
where $\eta_1, \eta_2 \sim O(1) $ are independent of $\ve $.
In fact, the reason why we take such a formula for $\Phi_2$ is as follows:
\begin{enumerate}
    \item  Since $  F_1(v), F_2(v)$ involves factors $ \ei{\pm }, \ei{\pm 3} $ and the linear operator $\ve^2\p_{tt}-\De +\frac1 {\ve^2}$ do not change phases, 
    thus $\Phi_2$ should at least involve the same phases.

    \item Only $\ei{\pm m}$ with $m$ odd appears in $\Phi_2$.
    This is because the NLKG equation \eqref{eq:u} is cubic and we have taken the forms $\ei{\pm} v$  in the first asymptotic expansion \eqref{first order expansion of u}.  
    
    \item The expansion of $\Phi_2$ stops at $\ei{\pm m}, m=3 $  which is enough  for our analysis.
    Indeed, if there is $\ei{ m}\eta_3 +c.c. $ with some $ m\ge 5 $ and $\eta_3 \sim O(1)$ in the formula of $\Phi_2$,
    then the linear term $\ve^4 \p_{tt}\Phi_2 -\ve^2 \De \Phi_2 +\Phi_2$ will introduce 
    the term $\ei{\pm m} (-m^2+1) \eta_3$ which is difficult to be canceled by other terms.   
\end{enumerate} 
According to the formula of $\Phi_2$, we can rewrite equation \eqref{second order equation} as follows:
\begin{align*}
    \ve^2 \p_{tt} r_1 -\De r_1 +\frac{1}{\ve^2} r_1
    +F_\ve
    =0,
\end{align*}
where $F_\ve$ is given by 
\begin{align*}
	F_\ve 
	&=\ei{} \ve^2 \big[
			2i\p_t \eta_1-\De \eta_1+3(2\eta_1|v|^2
			+v^2\ba \eta_1+ \ba v^2 \eta_2)
			+\p_{tt}v
		\big]
        +\ei{3}\big(v^3 -8\eta_2\big)
	\\
	&\quad
	+\ve^2\big[3\ei{5}v^2\eta_2
    +\ei{3}
		 \big(
			6i\p_t \eta_2-\De \eta_2+3(v^2\eta_1+2|v|^2\eta_2)
		\big)
	\big]\\
	&\quad
    +\ve^4 \big(\ei{}\p_{tt} \eta_1+\ei{3}\p_{tt}\eta_2 \big) 
    +3\ve^4 {\Phi_1}\Phi_2^2 +\ve^6 \Phi_2^3
    \\
    &\quad
    +
    3\left({\Phi_1}+\ve^2 {\Phi_2}\right)^2 {r_1}
    + 3\left({\Phi_1}+\ve^2 {\Phi_2}\right) r_1^2
    +r_1^3  
	+c.c..
\end{align*}
With the help of Lemmas \ref{lem:estimate of r with small perturbation}--\ref{lem:estimate of r with bootstrap}, 
we need to let $F_\ve $ be of a higher order than $O(\ve^2)$ such that we can prove $r_1$ is of a higher order than $O(\ve^2)$.
Due to Lemmas \ref{lem:estimate of r with small perturbation}--\ref{lem:estimate of r with bootstrap},
we neglect the dependence of time and find that the second to fourth lines of $F_\ve$ are of order $O(\ve^4)$ or higher, which is acceptable.
The first line of $F_\ve$ is of order $O(\ve^2)$ and needs to be eliminated.
As a consequence, we let 
\begin{align*}
	&2i\p_t \eta_1-\De \eta_1+3(2\eta_1|v|^2
	+v^2\ba \eta_1+ \ba v^2 \eta_2)
	+\p_{tt}v
	=0,\\
	&v^3 -8\eta_2=0.
\end{align*}
Let $\eta_1 =w, \eta_2=\frac1 8 v^3$, we obtain \eqref{Phi_2}, \eqref{expansion of u}. And the equation of $w$ reads
\begin{align*}
       2i\p_t w -\De w +\p_{tt} v +3 \bigg(\frac 18 |v|^4v+v^2 \ba w+ 2|v|^2 w \bigg)=0.
\end{align*}

$\bullet$ \textit{The equation for high-order error.} Owing to the analysis above, we obtain the equation of $r_1$
\begin{align*}
	\ve^2 \p_{tt} r_1 -\De r_1 +\frac{1}{\ve^2}r_1+H_1+H_2+H_3+H_r=0,        
\end{align*}
where
\begin{align*}
    H_1&= \ve^4\left[\frac18  \ei{3}\p_{tt} \left(v^3\right)
    + \ei{}\p_{tt}w  \right]+c.c.,  \\
    H_2&=
    \ve^2\bigg[
    \frac38 \ei{5}v^5
    +\frac 18 \ei{3}\left(
    6i\p_t \left(v^3\right) -\De \left( v^3\right)
    +3\left(2 |v|^2v^3 + 8 v^2 w \right) \right)
    \bigg]
    +c.c.,
    \\
    H_3&=
    3\ve^4 {\Phi_1}\Phi_2^2 +\ve^6 \Phi_2^3 ,
    \\
    H_r&=
    3\left({\Phi_1}+\ve^2 {\Phi_2}\right)^2 {r_1}
    + 3\left({\Phi_1}+\ve^2 {\Phi_2}\right) r_1^2
    +r_1^3.
\end{align*}

Now, it is time to set the initial data of $w, r_1$ finely.
By \eqref{Phi_2}--\eqref{eq:r}, we have the following formula for initial data: 
\begin{equation}
    \label{relationships}
    \begin{aligned}
        r(0)
        &=\ve^2\left( \frac{1}{8} \left(v_0^3+\ba {v_0}^3\right) 
        +w_0+\ba{w_0}\right) 
        +r_1(0)
        =0, \\
        \p_t r(0) 
        & = i\left[
        \frac{3}{8} \left(v_0^3-\ba {v_0}^3\right)
        +w_0-\ba{w_0}
        \right]
        + \ve^2\left[  \frac18 \p_t\big(
        v^3 + \ba v^3
        \big)(0)
        +\p_t (w+\ba w)(0)\right] 
        +\p_t r_1(0)  \\
        &=-\left(\p_t v(0) +\p_t \ba v(0) \right).	
    \end{aligned}
\end{equation}
To gain order $O(\ve^4)$ for $r_1$, 
in view of Lemma \ref{lem:estimate of r with small perturbation} and by imitating \eqref{3.5},
we need initial data $r_1(0), \p_t r_1(0)$ satisfy
\begin{align*}
   \ve^{2-s_c}  
   \nl{\p_t r_1(0)}{2}
   +\ve^{-s_c}\nl{\jd{\frac 12}r_1(0)}{2}
   \les \ve^{4-s_c}.
\end{align*} 
Thus we let 
\[     {r_1}(0)=0, \qquad
\p_t {r_1}(0) = -\ve^2 \left[
\frac18 \p_t\big(
v^3 + \ba v^3
\big)(0)
+\p_t (w+\ba w)(0)
\right]. \]
This and \eqref{relationships} give that 
\eqn{
    &\frac 18 \left( v_0^3 + \ba {v_0}^3\right)+w_0+\ba {w_0}= 0, \\
    &i\left[\frac 38\left(v_0^3-\ba {v_0}^3\right)
    +w_0-\ba {w_0}\right]=-\left(\p_t v(0) +\p_t \ba v(0) \right),}
which together with \eqref{eq:v} yields
\begin{align*}
    w_0&=\frac 14 \De (v_0-\ba {v_0}) 
    -\frac14 v_0^3+\frac 18 \ba {v_0}^3 
    -\frac{3}{4}|{v_0}|^2(v_0-\ba {v_0}).
\end{align*}
Hence we have equation \eqref{eq:f} for $w$ and the equation of  $r_1$ is given by
\eq{
        \label{eq:r_1}
        &\ve^2 \p_{tt} r_1 -\De r_1 +\frac{1}{\ve^2}r_1
        +H_1+H_2+H_3+H_r=0,
        \\
        &   {r_1}(0)=0, 
        \\
        &\p_t {r_1}(0)= -\ve^2 \left[
        \frac18 \p_t\big(
        v^3 + \ba v^3
        \big)(0)
        +\p_t (w+\ba w)(0)
        \right].
    }
         
\section{The key estimates}
\label{the key estimates}
\subsection{Estimation on $w$}
\label{estimate on f}
In this section, we establish the space-time estimate of $w$. 
Recall $w$ satisfies the following nonlinear Schr\"odinger equation:
\eqn{ 
    &2i\p_t w -\De w +\p_{tt} v +3 \left(\frac 18 |v|^4v+v^2 \ba w + 2|v|^2 w \right)=0,\\
    & w_0
    =\frac 14 \De (v_0-\ba {v_0})
    -\frac14 v_0^3+\frac 18 \ba {v_0}^3 
    -\frac{3}{4}|{v_0}|^2(v_0-\ba {v_0}),}
where $v$ satisfies \eqref{eq:v}.
The Duhamel formula of $w(t)$ is,
\[ w(t)=e^{-\frac i2 t\De }w_0+\frac i2\int_{0}^{t}e^{-\frac i2(t-s)\De}
\left[
    \p_{ss}v+3\left(\frac 18 |v|^4v+v^2  \ba w+ 2|v|^2 w \right)
\right](s)ds. \]

Before stating the lemma, we first make a simple observation about the space-time estimate of $ w$.
On the one hand, we require at least $H^2_x$ regularity of $v_0$ since initial data $w_0$ involves the term $\De v_0$.
On the other hand, the term $\partial_{tt}v$ in the nonlinearity implies $H^4_x$ regularity of $v_0$
which is the highest regularity requirement,
and the dependence of time after applying Strichartz estimates.
The forthcoming lemma will verify this observation in detail.

We give the space-time estimate of $w$ as follows.
\begin{lem}
    \label{lem:estimate of f}
    Let $d=2, 3$ and $ v_0 \in H^{d-2}_x(\R^d)$.
    Assume that $w$ and $v$ are solutions of \eqref{eq:f} and \eqref{eq:v}, respectively, 
    and $I=[0,T]$ is the corresponding maximal lifetime interval of $w$. 
    Then we have for $\ga \ge  -s_c$,
    \begin{align}
        \xsc{\naabs{\gamma}w}
        &\le C \big(
        \nh{{v_0}}{2+\ga+s_c}
        +T \nh{{v_0}}{4+\ga+s_c}
        \big),
    \label{est: ga w}
    \end{align}
    and
    \begin{align}
       \xsc{\naabs{\ga}\p_t w}
       &\le C\big(
       \nh{{v_0}}{4+\ga+s_c}
       +T \nh{{v_0}}{6+\ga+s_c}
       \big),
           \label{est:p_tf}
    \end{align}
    where the constant $C$ depends on $\norm{{v_0}}_{H^{d-2}_x}$.
\end{lem}
\begin{proof}
We prove \eqref{est: ga w} first.
 It follows from \eqref{eq:v} that
\begin{equation}
    \label{p_ttv}
     \p_{tt} v=O\left( \De^2 v +v^2\De v+ \De  \left(v^3\right)+v^5\right).
\end{equation} 
Owing to Duhamel formula, Lemma \ref{strichartz estimate for NLS}, and \eqref{p_ttv}, 
we have for $\ga \ge -s_c$,
\begin{align*}
   \xsc{\naabs{\gamma}w}
   &\les \nh{\naabs{\gamma}w_0 }{s_c}
        +\xsc{\naabs{\ga}\int_{0}^{t}e^{-\frac i2(t-s)\De}
            \left[
            \p_{ss}v+3\left(\frac 18 |v|^4v+v^2 \ba w+ 2|v|^2 w \right)
            \right](s)ds} \\
   &\les \nh{\naabs{\gamma}w_0 }{s_c}
        +\nth{\naabs{\ga+4} v}{1}{s_c}
        +\zsc{\naabs{\ga}(v^2\De v)}
        +\zsc{\naabs{\ga+2}(v^3)}\\
        &\quad 
        +\zsc{\naabs{\ga} (v^5)}
        +\zsc{\naabs{\ga}\left(v^2 w\right)}.
\end{align*}
By Lemma \ref{chain rules}, Remark \ref{rem:chain rules}, and the Gagliardo-Nirenberg inequality,
we obtain
\begin{align*}
    \xsc{\naabs{\gamma}w}
    &\les \nh{\naabs{\gamma}w_0 }{s_c}
    +T\nh{{v_0}}{4+\ga+s_c}
    +\sc{\naabs{\ga+2 }v}\sc v^2
    +{\chi_0}\sc{\naabs{\ga}v}\sc v \sc{\De v}\\
    &\quad
    +\sc{\naabs{\ga} v} \sc{v}^2\ltx{v}{\oo}^2
    +{\chi_0}\sc{\naabs{\ga}v}\sc v\sc{w}
    +\sc v^2\sc{\naabs{\ga}w},
\end{align*}
where ${\chi_0}$ is $0$ for $\ga \in [-s_c, 0]$ or $1$ for otherwise.   
Hence, by Lemma \ref{scattering theory} and the Gagliardo-Nirenberg inequality, we get for $\ga \ge -s_c$,
\begin{align*}
   \xsc{\naabs{\gamma}w} 
   &\les 
       \nh{\naabs{\ga }w_0}{s_c}
   + \nh{{v_0}}{2+\ga+sc}
   +T\nh{{v_0}}{4+\ga+s_c}
   +{\chi_0}\sc{\naabs{\ga}v}\sc{w}\\
   &\quad
   +\sc v^2\sc{\naabs{\ga}w}, 
\end{align*}
where the implicit constant depends on $\norm{{v_0}}_{H^{d-2}_x}$.

For $\ga \in [-s_c, 0]$, 
since $\sc{v}<\oo$,
a standard bootstrap process as shown in Lemma \ref{lem:estimate of r with bootstrap} gives that
\begin{align*}
   \xsc{\naabs{\gamma}w} 
   &\les 
       \nh{\naabs{\ga }w_0}{s_c}
   + \nh{{v_0}}{2+\ga+sc}
   +T\nh{{v_0}}{4+\ga+s_c}. 
\end{align*}
It follows from the definition of $w_0$ and Remark \ref{rem:chain rules} that 
    \begin{align*}
        \nh{\naabs{\ga }w_0}{s_c}
        \les \nh{v_0}{2+\ga+s_c}+\nh{ \naabs{\ga}(v_0^3)}{s_c}
        \les \nh{v_0}{2+\ga+s_c},
    \end{align*}
which implies
\begin{align*}
   \xsc{\naabs{\gamma}w} 
   &\les 
   \nh{{v_0}}{2+\ga+sc}
   +T\nh{{v_0}}{4+\ga+s_c}.
\end{align*}

For $\ga >0$, using the above estimate with $\ga=0$, we infer that 
\begin{align*}
    \xsc{\naabs{\gamma}w}
    &\les \nh{\naabs{\gamma}w_0 }{s_c}
    +\nh{{v_0}}{2+\ga+sc}
    +T\nh{{v_0}}{4+\ga+s_c}\\
    &\quad+\sc{\naabs{\ga}v}
    \left(\nh{{v_0}}{2+sc}+T\nh{{v_0}}{4+s_c}\right)
    +\sc v^2\sc{\naabs{\ga}w}.
\end{align*}
Using a similar bootstrap argument as shown in Lemma \ref{lem:estimate of r with bootstrap}, we obtain for $ \ga \ge -s_c$,
\begin{align*}
   \xsc{\naabs{\gamma}w}
   &\les 
        \nh{{v_0}}{2+\ga+s_c}
        +T\nh{{v_0}}{4+\ga+s_c},
\end{align*}  
which proves \eqref{est: ga w}.

Now it remains to show the estimate \eqref{est:p_tf}. 
Note by \eqref{eq:f}, we have
\begin{equation*}
    \p_t w=O \left( \De w +\p_{tt} v +v^5+v^2 w\right),
\end{equation*}
which gives
\begin{align*}
   \xsc{\naabs{\ga}\p_t w}
   &\les \xsc{\naabs{\ga+2}w}
   +\xsc{\naabs{\ga}\p_{tt}v}
   +\xsc{\naabs{\ga}v^5}
   +\xsc{\naabs{\ga}\left(v^2 w \right)} . 
\end{align*}
Due to Remark \ref{rem:chain rules} and \eqref{p_ttv}, we deduce that 
\begin{align}
      \xsc{\naabs{\ga}\p_{tt}v} 
      &\les 
       \xsc{\naabs{\ga+4}v}
       +\xsc{\naabs{\ga+2}v}\ltx{v}{\oo}^2
       +\xsc{\naabs{\ga}\left(v^2\De v\right)}
       +\xsc{\naabs{\ga}(v^5)} 
       \notag\\
      &\les
         \nh{{v_0}}{4+\ga+s_c}, 
         \label{est:ga p_ttv}
\end{align}
where the constant depends on $\norm{{v_0}}_{H^{d-2}_x}$.
Combining those estimates \eqref{est: ga w}, \eqref{est:ga p_ttv}, and invoking Remark \ref{rem:chain rules},  we establish for $ \ga \ge -s_c$,
\begin{align*}
    \xsc{\naabs{\ga}\p_t w}
    &\les  \xsc{\naabs{\ga+2}w}
    +\nh{v_0}{4+\ga+s_c} 
    +\ltx{v}{\oo}^2\xsc{\naabs{\ga}w}
    +\chi_0\ltx{\naabs{\ga} v}{\oo} \ltx{v}{\oo}\xsc{w} \notag\\
    &\les \nh{{v_0}}{4+\ga+s_c}
    +T \nh{{v_0}}{6+\ga+s_c},
\end{align*}
where the implicit constant depends on $\norm{v_0}_{H^{d-2}_x}$.
This finishes the proof of the lemma.
\end{proof}

\subsection{Useful estimate}
\label{estimate on error}
In this subsection, we shall develop a useful estimate for $r_1$.
To apply the Strichartz estimate for NLKG,  
we transform the equation of $r_1$ to a $\ve$-independent Klein-Gordon type equation.
Then, under the framework of Strichartz estimates as
stated in Remark \ref{rem:strichartz estimates},
we can establish the related useful estimate
for $r_1$ which will be applied to get 
the estimate of $r_1$ in next section.
The main procedure of the proof is similar to Lemmas \ref{lem:estimate of r with small perturbation}--\ref{lem:estimate of r with bootstrap}.

\subsubsection{Formulation}

We introduce the scaling transform $S_\ve $ 
as stated in Lemma \ref{lem:estimate of r with small perturbation} defined by
\[ S_\ve g(t,x):= \ve g(\ve^2 t, \ve x). \]
To eliminate the dependence of $\ve$ for equation \eqref{eq:r_1}, we set 
\begin{equation}
    \label{transform: R_1, h, ti f}
     ({R_1}, h, \ti w )(t, x)= S_\ve ({r_1}, v, \ve^2 w), 
\end{equation}
and for convenience, we denote
\begin{align*}
   \ti {\Phi_1}&= e^{it} h+ c.c.,\\
   \ti {\Phi_2}&=\frac 18 \left(
                e^{3it}h^3\right)
                +e^{it}\ti w +c.c..
\end{align*}
Here we note that $h$ and $ \ti w $ still satisfy \eqref{eq:v} and \eqref{eq:f} with scaling transformed initial data, respectively. 
Then, the equation \eqref{eq:r_1} is transformed to the following $\ve$-independent equation:
\begin{align}
    \label{eq:{R_1}}
    \p_{tt} {R_1} -\De {R_1}+{R_1}+{G_1}+{G_2}+{G_3}+{G_R}=0,
\end{align}
with the initial data 
\begin{align*}
    {R_1}(0)=0, \quad 
    \p_t {R_1}(0) = \ve^2 S_\ve (\p_t r_1(0)),
\end{align*}
where 
\begin{align}
    \label{G_1}
    {G_1}&=\frac18 e^{3it}\p_{tt} \left( h^3\right)
    + e^{it}\p_{tt}\ti w +c.c., \\
    \label{G_2}
    {G_2}&=\frac38 e^{5it}h^5
    +\frac 18 e^{3it}\left[
    6i\p_t \left(h^3\right) -\De \left( h^3\right)
    +3\left(2 |h|^2h^3 + 8h^2 { \ti w}\right) \right]
    +c.c.,\\
    \label{G_3}
    {G_3}&=3\ti {\Phi_1}\ti {\Phi_2}^2+\ti {\Phi_2}^3,\\
    \label{G_R}
    {G_R}&=3(\ti {\Phi_1}+\ti {\Phi_2})^2{R_1}+3(\ti {\Phi_1}+\ti {\Phi_2})R_1^2+R_1^3.
\end{align}

Now we make the wave decomposition as shown in the proof of Lemma \ref{lem:estimate of r with small perturbation}. 
Denote
\begin{align*}
   \phi = (\p_t +i \jd{})\jd{-1}{R_1}, 
\end{align*}
then we have
\eq{ \label{R1, phi}
    &{R_1}=\im \phi,\\
    &\p_t {R_1} =\re \jd{}\phi.}
We rewrite equation \eqref{eq:{R_1}} as follows: 
\begin{align}
    (\p_t -i \jd{})\phi =-\jd{-1}({G_1}+{G_2}+{G_3}+{G_R}),
    \label{eq:phi}
\end{align}
with initial data 
\begin{equation}
    \label{phi0}
    \phi(0):={\phi_0}=\jd{-1}\p_t {R_1}(0)=\ve^2\jd{-1}S_\ve (\p_t {r_1}(0)).
\end{equation}
The Duhamel formula reads
\begin{align*}
   \phi(t)=e^{it\jd{}}{\phi_0} 
               -\int_{0}^{t}e^{i(t-s)\jd{}}\jd{-1}({G_1}+{G_2}+{G_3}+{G_R})(s) ds.
\end{align*}

\subsubsection{Estimate of $\phi$}
\label{estimate on nonlinearity}
We give the estimate of $\phi$ as follows.
\begin{lem}
    \label{lem:estimate of phi}
    Let $d=2,3 $ and $\phi(t,x ): [0, T]\times \R^d \to \C$ satisfy equation \eqref{eq:phi} with the initial data \eqref{phi0}.
    Then for suitably small constant
    $\de_0$ satisfying
    \begin{align*}
           \ve^4T\nh{v_0}{4+s_c} \le C\left(\norm{v_0}_{H^{d-2}_x}
           \right) \de_0,  
    \end{align*}
     the following estimate is valid for $t\in [0, T]$ and $   \ga \in [-s_c, 0]$,
    \begin{align}
       &\ysc{\naabs{\ga}\int_{0}^{t}e^{i(t-s)\jd{}}\jd{-1}({G_1}+{G_2}+{G_3}+{G_R})(s) ds}
       \notag\\
        &\quad\les
        \ve^{4+\ga}     
        \Big(
        \nh{{v_0}}{4+\ga+s_c}
        +\ve^2T\nh{{v_0}}{6+\ga+s_c}
        +\ve^4T^2\nh{{v_0}}{8+\ga+s_c}
        \Big)
         \notag\\
        &\qquad
        + \left(
        \sc{h}  
        +\ve^2 \nh{v_0}{2+s_c}
        +\de_0             
        \right)^2
        \sc{\naabs{\ga}R_1},
        \label{est:nonlinearity}  
    \end{align}
    where the implicit constant depends on $\norm{v_0}_{H^{d-2}_x}$.
\end{lem}

\begin{proof}
    
We divide the proof into four parts: the estimates of $G_1, G_2, G_3$, and $ G_R$.

$\bullet$ Estimation on ${G_1}$ part.

By the definition \eqref{G_1} of $G_1$ and Lemma \ref{strichartz estimate},
we have 
\begin{align*}
    \ysc{\naabs{\ga}\int_{0}^{t}e^{i(t-s)\jd{}}\jd{-1}
    G_1ds}
    &\les \ysc{\naabs{\ga}\int_{0}^{t}e^{i(t-s)\jd{}}\jd{-1}\p_{ss}
        \left( h^3\right) (s)ds}
        \\
    &\quad 
        +\ysc{\naabs{\ga}\int_{0}^{t}e^{i(t-s)\jd{}}\jd{-1}\p_{ss} \ti  w (s)ds},
\end{align*}
then it suffices to deal with the $Y_{s_c}$ estimates for $\p_{tt} \left(  h^3\right), \p_{tt}{\ti w}$.

First, since $h$ satisfies \eqref{eq:v} with initial data $h(0):=h_0=S_\ve {v_0}$,
then we have the following equalities:
\begin{align}
    \label{p_t h}
    \p_{t} h&=O \left(\De h + h^3\right),
    \qquad
    \p_{tt} h=O\left( \De^2 h +h^2\De h+  \De \left( h^3\right) +h^5\right), \\ 
    \label{p_ttth}
    \p_{ttt}h
    &=O
    \left(\De^3 h+\De^2(h^3)+\left(\De h\right)^2h+h^2\De^2h+\De\left(h^2\De h\right)
    +h^4\De h+h^2\De (h^3) +\De (h^5)+h^7
    \right).
\end{align}
By Lemmas \ref{chain rules}, \ref{strichartz estimate}, and the Gagliardo-Nirenberg inequality,  
we obtain
\begin{align*}
    &\ysc{\naabs{\ga}\int_{0}^{t}e^{i(t-s)\jd{}}\jd{-1}\p_{ss}
        \left( h^3\right) (s)ds} \notag\\ 
    &\quad \les \zsc{\naabs{\ga} \left(h^2\p_{tt}h\right)}
                +\zsc{\naabs{\ga} \left(h(\p_t h)^2\right)}
        \\
    &\quad \les \sc{\naabs{\ga} \p_{tt}h} \sc h^2 
                +\sc{\naabs{\ga}\p_t h}\sc{\p_t h} \sc h.
\end{align*}
In light of \eqref{p_t h}, 
Lemma \ref{scattering theory}, 
and Remark \ref{rem:chain rules}, we get
\begin{align*}
   \sc{\naabs{\ga}\p_t h}
   &\les \sc{\naabs{\ga+2}h}
            +\sc{\naabs{\ga} \left(h^3\right)} 
   \\
   &  \les \nh{h_0}{2+\ga+s_c}      
   \les \ve^{2+\ga}\nh{v_0}{2+\ga+s_c},
\end{align*}
and 
\begin{align*}
   \sc{\naabs{\ga}\p_{tt} h}
   &\les \sc{\naabs{\ga+4}h}
        +\sc{\naabs{\ga+2}h^3}
        +\sc{\naabs{\ga}\left(h^2\De h\right)}
        +\sc{\naabs{\ga}h^5}
    \\
   & \les \nh{h_0}{4+\ga+s_c}  
   \les \ve^{4+\ga}\nh{v_0}{4+\ga+s_c}.
\end{align*}
Hence, it turns out from Lemma \ref{scattering theory} that 
\begin{align}
       \ysc{\naabs{\ga}\int_{0}^{t}e^{i(t-s)\jd{}}\jd{-1}\p_{ss}
        \left( h^3\right) (s)ds}  
       \les \ve^{4+\ga}\nh{{v_0}}{4+\ga+s_c}.
   \label{est: {G_1}1} 
\end{align}
We thus finish the estimate about the first component of $G_1$.

We turn to deal with the term $\p_{tt} {\ti w} $.
Now that $\ti w $ satisfies \eqref{eq:f} with initial data 
$\ti w(0):=\ti w_0 = \ve^2 S_\ve w_0$,  we have the following relation:
\begin{align}
    \p_{tt} \ti w &=O\left(
                        \De^2 \ti w 
                        +\De^3 h 
                        +\De 
                            \left(
                            \De (h^3) +h^2\De h+h^5+h^2 \ti w
                        \right)
                        +\p_{ttt}h 
                        +h^4\p_t h
                        +\p_t\left( h^2\ti  w\right)
                        \right).
    \label{p_tt ti f}
\end{align}
Then it follows from \eqref{p_t h}, \eqref{p_ttth}, \eqref{p_tt ti f}, and
Lemma \ref{strichartz estimate} that
\begin{align}
\label{estimate of p_ss ti w}
   & \ysc{\naabs{\ga}\int_{0}^{t}e^{i(t-s)\jd{}}\jd{-1}\p_{ss}
         \ti w  (s)ds}\notag\\
         &\quad \les 
         \nth{\naabs{\ga+4} \ti w }{1}{s_c}
           +\nth{\naabs{\ga+6} h}{1}{s_c}
           +\zsc{\naabs{\ga+2} \left(
               h^2 \ti w
               \right)} 
               +\zsc{\naabs{\ga}\p_t\left(h^2 \ti w\right)} 
    \\ 
    \label{estimate of p_ss ti w-1}
          &\qquad      
          +\zsc{\naabs{\ga}
            \left(
                \De^2(h^3)+\left(\De h\right)^2h+h^2\De^2h+\De\left(h^2\De h\right)
            \right)}
    \\
    \label{estimate of p_ss ti w-2}
     &\qquad
        +\zsc{\naabs{\ga}\left(h^4\De h+h^2\De (h^3) +\De (h^5)+h^7\right)}
\end{align}
According to Lemmas \ref{chain rules}, \ref{strichartz estimate}, 
and Remark \ref{rem:chain rules}, we obtain that 
\begin{align}
   \eqref{estimate of p_ss ti w-1}+\eqref{estimate of p_ss ti w-2}
        &
     \les \sc{\naabs{\ga+4} h}\sc h^2
            +\sc{\naabs{\ga+2}h}\sc{\De h}\sc{h}
            \notag\\
   &\quad
    + \sc{\naabs{\ga+2}h}\sc{h}^2\ltx{h}{\oo}^2
    + \sc{\naabs{\ga}h}\sc{h}^2\ltx{h}{\oo}^4    
            \notag\\
   &\les \nh{h_0}{4+\ga+s_c}  \les \ve^{4+\ga}\nh{v_0}{4+\ga+s_c},
   \label{4.20+4.21}
\end{align}
where the constant depends on $\norm{v_0}_{H^{d-2}_x}.$
On account of Lemmas \ref{chain rules} and \ref{strichartz estimate},
we have,
\begin{align}
 \eqref{estimate of p_ss ti w}
    & \les 
         T
         \left(
            \nth{\naabs{\ga+4} \ti w }{\oo}{s_c}
                   +\nth{\naabs{\ga+6} h}{\oo}{s_c}
         \right)
        +\sc{h}^2 
        \left( \sc{\naabs{\ga+2}\ti w} +\sc{\naabs{\ga}\p_t \ti w}\right) 
        \notag\\
         &\quad  
        +\left( \sc{\naabs{\ga+2} h} 
        +\sc{\naabs{\ga}\p_t h }\right) \sc h \sc{\ti w}
       \label{p_tt ti f1}
\end{align}
To estimate the above inequality further, we need the estimate of $\ti w $.
By \eqref{transform: R_1, h, ti f} and Lemma \ref{lem:estimate of f}, we have for $\ga \ge -s_c,$ 
\begin{align*}
    \norm{\naabs{\gamma} \ti w}_{X_{s_c}([0, T])}
    &\les \ve^{2+\ga}
    \left(
        \nh{{v_0}}{2+\ga+s_c}
        + \ve^{2} T \nh{{v_0}}{4+\ga+s_c}
    \right),
    \\
    \norm{\naabs{\gamma} \p_t \ti w}_{X_{s_c}([0, T])}
    &\les \ve^{4+\ga}
    \left(
        \nh{{v_0}}{4+\ga+s_c}
        + \ve^{2} T \nh{{v_0}}{6+\ga+s_c}
    \right).
\end{align*}
Inserting these estimates for $\ti w, \p_t \ti w$ into \eqref{p_tt ti f1}
and invoking Lemma \ref{scattering theory}, we obtain
\begin{align}
   \eqref{estimate of p_ss ti w}
   &\les 
   \ve^{4+\ga}
   \left(
   \nh{{v_0}}{4+\ga+s_c}
   +\ve^2T  
   \nh{{v_0}}{6+\ga +s_c}
   +\ve^4T^2   
   \nh{{v_0}}{8+\ga+s_c}
   \right).
 \label{est:4.19}
\end{align}
Gathering estimates \eqref{4.20+4.21} and \eqref{est:4.19}, we have
\begin{align}
\label{est:G_12}
    &\ysc{\naabs{\ga}\int_{0}^{t}e^{i(t-s)\jd{}}\jd{-1}\p_{ss}
         \ti w  (s)ds}
        \notag \\
    &\qquad
    \les 
    \ve^{4+\ga}
   \left(
   \nh{{v_0}}{4+\ga+s_c}
   +\ve^2T  
   \nh{{v_0}}{6+\ga +s_c}
   +\ve^4T^2   
   \nh{{v_0}}{8+\ga+s_c}
   \right).
\end{align} 

Combining estimates \eqref{est: {G_1}1} with
 \eqref{est:G_12}, we establish 
\begin{align}
    &\ysc{\naabs{\ga}\int_{0}^{t}e^{i(t-s)\jd{}}\jd{-1}G_1(s)ds}
       \notag\\
        &\qquad
        \les 
        \ve^{4+\ga}
        \left(
        \nh{{v_0}}{4+\ga+s_c}
        +\ve^2T  
        \nh{{v_0}}{6+\ga +s_c}
        +
        \ve^4T^2   
        \nh{{v_0}}{8+\ga+s_c}
        \right),
        \label{est: {G_1}}
\end{align}
where the implicit constant depends on $\norm{v_0}_{H^{d-2}_x}$.
This ends the estimate of ${G_1}$.

$\bullet$ Estimation on ${G_2}$ part.

Since the definition \eqref{G_2} of $G_2$ only involves high oscillations $\ei{\pm m}$ with $m>1$, we can argue similarly as in the proof of Lemma \ref{lem:estimate of r with high oscillation}. 

First, using high-low frequency decomposition, we have 
\begin{align}
    \int_{0}^{t} e^{i(t-s)\jd{}}\jd{-1} {G_2} (s)ds 
    &= \int_{0}^{t} e^{i(t-s)\jd{}}\jd{-1} P_{ \le 1}{G_2} (s)ds  
    \label{{G_2}1}\\
    &\quad 
    +\int_{0}^{t} e^{i(t-s)\jd{}}\jd{-1}P_{>1} {G_2} (s)ds.
    \label{{G_2}2} 
\end{align}
Then the estimate of $G_2$ reduces to the estimates of \eqref{{G_2}1} 
and \eqref{{G_2}2}.

To estimate \eqref{{G_2}1}, 
it suffices to consider the following two terms:
\begin{equation}
\label{G_211}    
P_{\le 1} \int_{0}^{t} e^{i(t-s)\jd{}}\jd{-1} \left(
e^{5is} h^5(s) \right)ds, 
\end{equation}
and 
\begin{equation}
\label{G_212} 
  P_{\le 1} \int_{0}^{t} e^{i(t-s)\jd{}}\jd{-1} 
 e^{3is}
\left[
6i\p_s\left(  h^3 \right) -\De \left(  h^3\right)
+ 3\left(2|h|^2h^3 + 8 h^2\ti w \right) 
\right](s)ds. 
\end{equation}
For the first term \eqref{G_211}, integration by parts shows that
\begin{align*}
    \eqref{G_211}
    &=\frac{\jd{-1}} {i(5-\jd{})} P_{\le 1}e^{it \jd{}}\int_{0}^{t} h^5(s) de^{i(5-\jd{})s}\notag\\
    & 
    =\frac{\jd{-1}}{i(5-\jd{})}P_{\le 1 }
           \left[
                e^{5it}h^5(t)-e^{it\jd{}}h_0^5
                -\int_{0}^{t}e^{i(t-s)\jd{}+5is}\p_{s} \left(h^5(s)\right)ds
           \right].
\end{align*}
Owing to Lemmas \ref{chain rules}, \ref{strichartz estimate}, \ref{scattering theory}, and Remark \ref{rem:chain rules},  
we establish
\begin{align}
    \ysc{\naabs{\ga} \eqref{G_211}}
    &\les \nh{\jd{-\frac12}\naabs{\ga}(h^5_0)}{s_c}
            +\ysc{\jd{-1}\naabs{\ga}(h^5)}
            +\zsc{\naabs{\ga}\left( h^4\p_t h\right) } 
            \notag \\
    &\les \xsc{\naabs{\ga} (h^5)} 
            +\zsc{\naabs{\ga}\left( h^4(\De h+h^3)\right) } \notag\\
    &\les \ve^{4+\ga}
            \nh{{v_0}}{4+\ga +s_c},
    \label{est:4.29}
\end{align}
where the implicit constant depends on $\norm{v_0}_{H^{d-2}_x}$.
This finishes the estimate of \eqref{G_211}.
It remains to show the estimate of \eqref{G_212}.
Similarly, we get
\begin{align*}
    \eqref{G_212}
    =\frac{\jd{-1}}{i(3-\jd{})}P_{\le 1}
           \bigg\{ \bigg[&
                e^{3it}
                \left(
                    6i\p_t\left(  h^3 \right) 
                    -\De \left(  h^3\right)
                    + 3\left(2 |h|^2h^3 +8h^2\ti w\right)
                \right)	\notag\\
              &
              - e^{it\jd{}} \left(
              6i\p_t\left(  h^3 \right)(0) 
              -\De \left(  h_0^3\right)
              + 3\left( 2|h_0|^2h_0^3 + 8h_0^2\ti w_0\right)
              \right)
              \bigg]	\notag\\             
    &
     -\int_{0}^{t}e^{i(t-s)\jd{}}
         e^{3is} \p_s
     \left[
        6i\p_s\left(  h^3 \right) -\De \left(  h^3\right)
        + 3\left( 2|h|^2h^3 +8 h^2\ti w \right) 
     \right](s)ds
     \bigg\}.
\end{align*}
Invoking Lemmas \ref{chain rules}, \ref{strichartz estimate},
 \ref{scattering theory}, \ref{lem:estimate of f},
 and \eqref{est: {G_1}1},
  we find that 
\begin{align}
   &\ysc{\naabs{\ga} \eqref{G_212}}\notag \\
   &\les \xsc{\naabs{\ga} \p_t \left(h^3\right)}
            +\xsc{\naabs{\ga}\De (h^3)}
            +\xsc{\naabs{\ga}\left(|h|^2h^3\right)}
            +\xsc{\naabs{\ga}\left(h^2\ti w \right)}
            \notag \\
            &\quad
            +\zsc{\naabs{\ga}\p_{tt}( h^3)}
            +\zsc{\naabs{\ga+2}\p_t \left(h^3 \right)}
            +\zsc{\naabs{\ga}\p_t \left(|h|^2h^3\right)}
            +\zsc{\naabs{\ga} \p_t\left(h^2\ti w \right)}
            \notag\\          
   &\les \ve^{4+\ga}
   \left(
   \nh{{v_0}}{4+\ga+s_c}
   +\ve^2T \nh{{v_0}}{6+\ga+s_c}
   \right),
   \label{est:4.30}
\end{align}
where the implicit constant depends on $\norm{v_0}_{H^{d-2}_x}.$
The estimate of $\eqref{G_212}$ is completed.
Collecting estimates \eqref{est:4.29} and \eqref{est:4.30}, we obtain 
\begin{equation}
    \label{est: {G_2}1}
    \ysc{\naabs{\ga}\eqref{{G_2}1}}
    \les \ve^{4+\ga}
    \left(
    \nh{{v_0}}{4+\ga+s_c}
    +\ve^2T \nh{{v_0}}{6+\ga+s_c}
    \right).  
\end{equation}  
This finishes the estimate of \eqref{{G_2}1}.

Now we are in a position to estimate \eqref{{G_2}2}.
The proof is straightforward.
Thanks to Lemmas \ref{chain rules} and \ref{strichartz estimate}, we have
\begin{align*}
    &\ysc{\naabs{\ga}\eqref{{G_2}2}}
    \\
    &\les \zsc{\naabs{\ga+2}(h^5)}
            +\zsc{\naabs{\ga+2}\p_t\left(h^3\right)}
            +\zsc{\naabs{\ga+2}\left(h^2\ti w\right)}
            +\zsc{\naabs{\ga+4} \left(h^3\right)}
            \\
    &\les \sc{\naabs{\ga+2}h}
                \left(
                    \sc h^2\ltx{h}{\oo}^2
                    +\sc{\p_t h}\sc h
                    +\sc h\sc {\ti w}
                \right)
                \\
                &\quad
            +
            \left(
            \sc{\naabs{\ga+2}\p_th }
            +\sc{\naabs{\ga+2}\ti w}
            +\sc{\naabs{\ga+4}h}
            \right)
            \sc h^2 
\end{align*}
By Lemmas \ref{scattering theory}, \ref{lem:estimate of f}, and the
Gagliardo-Nirenberg inequality, it transpires that
\begin{align*}
    \ysc{\naabs{\ga}\eqref{{G_2}2}}    
   &\les \ve^{4+\ga}
   \left(
   \nh{{v_0}}{4+\ga+s_c}
   +\ve^2 T \nh{{v_0}}{6+\ga+s_c}
   \right). 
\end{align*}
This together with \eqref{est: {G_2}1} gives 
\begin{equation}
    \label{est: {G_2}}
    \ysc{  \naabs{\ga}  \int_{0}^{t} e^{i(t-s)\jd{}}\jd{-1} {G_2} (s)ds }
    \les\ve^{4+\ga}
    \left(
    \nh{{v_0}}{4+\ga+s_c}
    +\ve^2T \nh{{v_0}}{6+\ga+s_c}
    \right),
\end{equation}
where the implicit constant depends on $\norm{v_0}_{H^{d-2}_x}$.
Thus we give the estimate of $G_2$.

$\bullet$ Estimation on ${G_3}$ part.

By the definition \eqref{G_3} of $G_3$, Lemma \ref{strichartz estimate}, and Remark \ref{rem:chain rules},
we have for $   \ga \in [-s_c, 0]$
\begin{align}
\label{est: G_3-1}
    &
    \ysc{\naabs{\ga}\int_{0}^{t} e^{i(t-s)\jd{}}\jd{-1} {G_3} (s)ds}\notag\\
    &\quad\les 
       \sc h \sc{\ti {\Phi_2}}\sc{\naabs{\ga}\ti {\Phi_2}}
       +\sc{\ti {\Phi_2}}^2 \sc{\naabs{\ga}\ti {\Phi_2}}.
\end{align}
In light of Lemmas \ref{scattering theory} and \ref{lem:estimate of f}, we have for $\ga \ge -s_c$,
\begin{align}
    \sc{\naabs{\ga} \ti {\Phi_2}}
    &\les \sc{\naabs{\ga} \left( h^3\right)}
            +\sc{\naabs{\ga} \ti w }
           \notag \\
     &\les \ve^{2+\ga}
                \left(
                \nh{{v_0}}{2+\ga+s_c}
                +\ve^2T\nh{{v_0}}{4+\ga+s_c}
                \right).
           \label{est:ti {Phi_2}}
\end{align}
where the implicit constant depends on $\norm{v_0}_{H^{d-2}_x}$.
Inserting \eqref{est:ti {Phi_2}} into \eqref{est: G_3-1}, we infer that 
\begin{align}
	&
    \ysc{\naabs{\ga}\int_{0}^{t} e^{i(t-s)\jd{}}\jd{-1} {G_3} (s)ds}\notag\\
    &\quad\les 
    \ve^{4+\ga}
    \Big(
    \nh{{v_0}}{4+\ga+s_c}
    +\ve^2T\nh{{v_0}}{6+\ga+s_c}
    +\ve^4T^2\nh{{v_0}}{8+\ga+s_c}
    \Big)
    \notag\\
    &\qquad\qquad
    \cdot
    \Big[
           1+
    \ve^2
    \left(
    \nh{{v_0}}{2+s_c}
    +\ve^2T\nh{{v_0}}{4+s_c}
    \right)
    \Big].
    \label{est:{G_3}}
\end{align}
where the implicit constant depends on $\norm{v_0}_{H^{d-2}_x}$.
This completes the estimate of ${G_3}$.

$\bullet$ Estimation on ${G_R}$ part.

By Lemmas \ref{chain rules}, \ref{strichartz estimate}, \ref{scattering theory}, and \eqref{est:ti {Phi_2}}, 
we have
\begin{align}
    &\ysc{\naabs{\ga}\int_{0}^{t} e^{i(t-s)\jd{}}\jd{-1} {G_R} (s)ds}
    \notag\\
    &\quad\les \sc{\left(
        \ti {\Phi_1} +\ti {\Phi_2}
        \right)}^2 \sc{\naabs{\ga}{R_1}}
    +\sc{ \left(
        \ti {\Phi_1} +\ti {\Phi_2}
        \right) }\sc{{R_1}} \sc{\naabs{\ga}{R_1}}
        \notag \\
        &\qquad
    +\sc{{R_1}}^2\sc{\naabs{\ga}{R_1}}
    \notag \\
    &\quad\les \left[
    \sc h +\ve^2 
    \left(
    \nh{{v_0}}{2+s_c}
    +\ve^2T\nh{{v_0}}{4+s_c}
    \right)               
    \right]^2\sc{\naabs{\ga}{R_1}} 
    \notag\\
    &\qquad
    +\left[
    \sc h +\ve^2 
    \left(
    \nh{{v_0}}{2+s_c}
    +\ve^2T\nh{{v_0}}{4+s_c}
    \right)               
    \right]\sc{{R_1}} \sc{\naabs{\ga}{R_1}}
    \notag\\
    &\qquad
    +\sc{{R_1}}^2\sc{\naabs{\ga}{R_1}}.
    \label{est:{G_R}1}
\end{align}
To avoid the complicated analysis, we use the following estimate of $R_1$:
\begin{align}
   \label{est:R_1 s_c}
    \sc{R_1}
   &\le C\left(\norm{v_0}_{H^{d-2}_x}\right) 
        \left(\ve^2+\ve^4T\nh{{v_0}}{4+s_c}\right).
\end{align} 
Indeed, by \eqref{first order expansion of u} and \eqref{expansion of u}, we have
\begin{align*}
    R_1=R-\ti{\Phi_2}.
\end{align*}
Since the estimate of $\ti \Phi_2$ is given in \eqref{est:ti {Phi_2}}, we need the estimate of $\sc{R}$
which has been obtained in \cite{lei}.
Recall in \cite{lei}, 
under the condition for $T$ satisfying that
for small $\de_0$,
\begin{equation*}
    \ve^4T\nh{v_0}{4+s_c} \le C\left(\norm{v_0}_{H^{d-2}_x}\right) \de_0, 
\end{equation*}
then there holds that
\begin{align*}
   \sc{R}
   &\le C\left(\norm{v_0}_{H^{d-2}_x}\right) 
   \left(
        \ve^2+\ve^4T\nh{{v_0}}{4+s_c}
   \right).
\end{align*}
The above estimate and \eqref{est:ti {Phi_2}} imply \eqref{est:R_1 s_c}.
Inserting \eqref{est:R_1 s_c} into \eqref{est:{G_R}1},
we obtain 
\begin{align}
   \ysc{\naabs{\ga}\int_{0}^{t} e^{i(t-s)\jd{}}\jd{-1} {G_R} (s)ds}
       &\les 
       \left(
       \sc{h} 
       +\ve^2\nh{v_0}{2+s_c}
       +\de_0               
       \right)^2
       \sc{\naabs{\ga}R_1},
       \label{est:{G_R}}
\end{align}
where the implicit constant is dependent of  $\norm{v_0}_{H^{d-2}_x}$.
Thus we establish the estimate of $G_R$.

Collecting those estimates \eqref{est: {G_1}}, \eqref{est: {G_2}}, \eqref{est:{G_3}},
and \eqref{est:{G_R}}, 
we complete the proof of this lemma.
\end{proof}

\section{Non-relativistic limit}
\label{Non-relativistic limit}
In this section, we prove Theorem \ref{main theorem}.
The proof is divided into two parts: the regular case and the non-regular case.

\subsection{Regular case}

\label{the regular case}
In this subsection, we establish the proof of Theorem \ref{main theorem} for $H^8$-data.  
We divide the time interval into two distinct ranges:
a short-time range and a long-time range, 
which will be addressed in different manners. 
For the short-time range, 
the estimate of $\phi$ established in the previous section is valid.
For the long-time range, we have to seek the help of conservation law.
Combining both cases, we ultimately establish the convergence rate estimate of the error term $r_1$ for $H^8$-data. 

\begin{prop}
    \label{main prop}
    Let $d=2,3$ and $u_0,u_1 \in H^{8}_x(\R^d)$.
    Assume that $\Phi_1$ and $\Phi_2$ satisfy \eqref{Phi_1} and \eqref{Phi_2},
    and $v$, $w$, and $r_1$ satisfy equations \eqref{eq:v}, \eqref{eq:f}, and
    \eqref{eq:r_1}, respectively.
    Then we have for all $t\ge 0$,
    \[ \nl{u^\ve -\Phi_1 -\ve^2 \Phi_2}{2}
    \le C
    \ve^{4} \big(1+t\nh{v_0}{6}+t^2\nh{v_0}{8}\big),
    \]
    where the  constant $C$  depends on $\norm{u_0}_{H_x^{4}}$ and
    $\norm{u_1}_{H_x^{4}}$.
\end{prop}

\begin{proof}
In view of Lemma \ref{lem:estimate of phi},
we divide time interval $I=[0, T]$ into the short and long time regimes, respectively,
\[  C_0\ve^4T \nh{v_0}{4+s_c} \le \de_0  \mbox{\quad 
	and \quad}  
C_0\ve^4T \nh{v_0}{4+s_c} > \de_0,
\]
where the constant $C_0$ is the reciprocal of $C\big(\norm{v_0}_{H^4_x}\big)$ in Lemma \ref{lem:estimate of phi}.

Case 1(short time): $C_0\ve^4T \nh{v_0}{4+s_c}\le \de_0$.
In this case, by Lemma \ref{lem:estimate of phi},
we get for $\ga=-s_c$,
\begin{align}
    \yo{ \phi}
    &\le 
    C_1\nl{\jd{\frac12}{\phi_0}}{2}
    +C_2
    \ve^{4-s_c}  
    \Big(
    1
    +\ve^2T\nh{{v_0}}{6}
    +\ve^4T^2\nh{{v_0}}{8}
    \Big)
    \notag\\
    &\quad
    + C_3\left(
    \sc{h}
    +\ve^2+\de_0             
    \right)^2
    \so{R_1},
    \label{est:phi-1}
\end{align}
where constants $C_1$, $C_2$ and $C_3$ depend on $\norm{v_0}_{H^{4}_x}.$
By $\sc{h_0}<\oo$ and \eqref{R1, phi}, we argue by bootstrap as shown in Lemma \ref{lem:estimate of r with bootstrap} and obtain for $\ve,  \de_0$ small,
\begin{align}
 \norm{ \phi}_{Y_0(I)}
&\les 
\nl{\jd{\frac12}{\phi_0}}{2}
+
    \ve^{4-s_c}
\Big(
1
+\ve^2T\nh{{v_0}}{6}
+\ve^4T^2\nh{{v_0}}{8}
\Big).
     \label{est:ga is neq 0}
 \end{align}
 By virtue of \eqref{eq:r_1}, \eqref{est:p_tf}, \eqref{phi0}, 
 and Remark \ref{rem:chain rules}, we have for $\ga \in [-s_c, 0],$
 \begin{align}
     \nh{\naabs{\ga}\jd{\frac12}{\phi_0}}{s_c}   
     &\les \ve^{2+\ga}\nh{S_\ve\left(\naabs{\ga}\p_t {r_1}(0)\right)}{s_c}\notag \\
     &\les \ve^{4+\ga}\left(\nh{\naabs{\ga}\p_t (v^{3})(0)}{s_c}+\nh{\naabs{\ga}\p_t w(0)}{s_c}\right)
     \notag \\
     &\les \ve^{4+\ga}
     \nh{{v_0}}{4+\ga+s_c},
     \label{est:phi0}
 \end{align}
 where the constant depends on $\norm{v_0}_{H^{d-2}_x}$.
The estimate \eqref{est:phi0} together with \eqref{R1, phi} and \eqref{est:ga is neq 0} yields that
\begin{align*}
   \norm{R_1}_{L^\oo_tL^2_x([0, T])}
   &\les
   \norm{\phi}_{L^\oo_tL^2_x([0, T])}
   \les
    \ve^{4-s_c}
\Big(
1
+\ve^2T\nh{{v_0}}{6}
+\ve^4T^2\nh{{v_0}}{8}
\Big).
\end{align*} 
Applying the scaling transform backward, we establish that
\begin{align}
    \norm{r_1}_{L^\oo_tL^2_x([0, T])}
    =
    \ve^{s_c}
    \norm{R_1}_{L^\oo_tL^2_x([0, \ve^{-2 }T])}
   \les
    \ve^{4}
    \Big(
    1
    +T\nh{{v_0}}{6}
    +T^2\nh{{v_0}}{8}
    \Big).
  \label{est:r_1 case 1}
\end{align}
We thus complete the proof of Proposition \ref{main prop} for the short-time range.

Case 2(long time): $C_0\ve^4T \nh{v_0}{4+s_c} > \de_0$.
Thanks to the energy conservation law of equation \eqref{eq:u}, we have
\begin{align*}
    &\int
    \left(
    \left|\ve^{2}\p_{t}u^{\ve}\right|^2
    +\left|\na u^{\ve}\right|^2 
    +\ve^{-2}\left|u^{\ve}\right|^2
    +\frac{1}{2}\left|u^{\ve}\right|^4
    \right)dx\\
    &\quad=\int
    \Big[
    \ve^{-2}
        \left(
        \left| u_0\right|^2
        +\left| u_1\right|^2
        \right)
    +\left|\na u_0\right|^2 
    +\frac{1}{2}\left|u_0\right|^4
    \Big]dx,
\end{align*}
which implies
\[\ntx{u^{\ve}}{\oo}{2}\les\norm{\left(u_0,u_1\right)}_{H^{1}\times L^{2}}.\]
By the definition of $r_1$ and Lemma \ref{lem:estimate of f},
we obtain
\begin{align}
    \norm{r_1}_{L^\oo_tL^2_x([0, T])}
   &\les
   \norm{\left(u_0,u_1\right)}_{H^{1}\times L^{2}}
   +\ve^{2}
   \left(
   \norm{v^3}_{L^\oo_tL^2_x([0, T])}
   +\norm{w}_{L^\oo_tL^2_x([0, T])}
   \right)
   \notag\\
   &\les
   \norm{\left(u_0,u_1\right)}_{H^{1}\times L^{2}}
   +\ve^{2}
   \left(
   1+\nh{{v_0}}{2}
   +T\nh{{v_0}}{4}
   \right)
   \notag\\
   &\les 
   1+\ve^2\left(1+T\right),
    \label{est:r_1 case 2}
\end{align}
which implies \eqref{est:r_1 case 1} also holds for the long time.

Combining \eqref{est:r_1 case 1} and \eqref{est:r_1 case 2}, we prove Proposition \ref{main prop}.
\end{proof}

\subsection{Non-regular case}
\label{the nonregular case}
In this subsection, we aim to reduce the regularity requirement for initial data $v_0$
in Proposition \ref{main prop}. 
A natural approach is to decompose initial data $v_0$ into the high frequency $P_{> N}v_0$ and the low frequency $P_{\le N}v_0$, then we need to analyze the estimates of the corresponding solutions. 
On the one hand, the initial data $P_{\le N}v_0$ is smooth, and thus we can apply the result of Proposition \ref{main prop}.
On the other hand, 
we need to establish the estimates of solutions related to initial data $P_{> N}v_0$.
Combining both estimates of solutions with initial data $P_{\le N}v_0$ and $P_{> N}v_0$, we can choose an adapted frequency $N$ to gain the optimal convergence rate for $r_1$.
As we shall see, the limitation $v_0 \in H^\al, \al \ge 4 $ is needed for our discussion when addressing the estimate of $w$ related to initial data $P_{>N}w_0$.

\subsubsection{Estimate of $v^N$}
Suppose that $v_0\in H^\al_x, \al\in [4,8]$.
We embark on the high-low frequency decomposition for the initial data $v_0$,
and then we aim to get the $X_{s_c}$ norm estimate of solutions related to the high-frequency initial data $P_{>N} {v_0}$. 

First, we denote $v_N$ to be the solution of the following equation:
\eq{\label{eq:v_N} 
    &2i \p_t v_N-\De v_N =-3|v_N|^2v_N,\\
    & v_N(0)=P_{\le N} {v_0} .}
Denote $v^{N}=v-v_N$, then $v^{N}$ is the solution to the following equation:
\eq{\label{eq:v^N} 
    &2i \p_t v^{N}-\De v^{N} =-3(|v|^2v-|v_N|^2 v_N),\\
    & v^{N}(0)=P_{>N} {v_0},}
and the corresponding Duhamel formula reads
\[v^{N}(t)
    =e^{-\frac i2 t\De }P_{>N} {v_0}
     +\frac{3i}{2}\int_{0}^{t}e^{-\frac i2(t-s)\De}
\left(
|v|^2v-|v_N|^2 v_N
\right)(s)ds.
\]
By Lemmas \ref{strichartz estimate} and \ref{scattering theory}, we have for $\ga\geq -s_c$,
\begin{align}
       \xsc{\naabs{\ga}v^N}
       &\les\nh{\naabs{\ga}P_{>N} {v_0}}{s_c}+\zsc{\naabs{\ga}\left(|v|^2v-|v_N|^2 v_N\right)}
       \notag\\
        &\les
    \nh{P_{>N}{v_0}}{\ga+s_c}
    +\sc{\naabs{\ga}v^N}\sc{v}^2
    +\chi_0\sc{\naabs{\ga}v}\sc{v}\sc{v^N},
    \label{add v^N}
\end{align}
where $\chi_0=0$ when $ -s_c \le \ga \le 0$, and $\chi_0=1$ otherwise.

For $\ga\in [-s_c, 0]$, we have 
\[\xsc{\naabs{\ga}v^N}\les\nh{P_{>N} {v_0}}{\ga+s_c}+\sc{\naabs{\ga}v^N}\sc{v}^2.\]
Since $v_0\in H^\al_x$ with $\al\in [4,8]$, we have  
\[\sc{v}\le  C\left(\norm{v_0}_{H^{d-2}_x}\right)\nh{{v_0}}{s_c}< \oo.\]
Then, by a standard bootstrap argument as shown in Lemma \ref{lem:estimate of r with bootstrap}, we obtain for $\ga\in [-s_c, 0]$,
\begin{equation}
    \label{est:|na|^ga_0 v^N}
    \xsc{\naabs{\ga}v^N}\les\nh{P_{>N}{v_0}}{\ga+s_c}.
\end{equation}

For $\ga>0$, in view of \eqref{add v^N} and \eqref{est:|na|^ga_0 v^N} with $\ga=0$, we obtain
\begin{align*}
    \xsc{\naabs{\ga}v^N}
    &\les
    \nh{P_{>N}{v_0}}{\ga+s_c}
    +\sc{\naabs{\ga}v^N}\sc{v}^2
    +\nh{{v_0}}{\ga+s_c}\nh{P_{>N}{v_0}}{s_c}.
\end{align*}
Since $\sc{v}<\oo$, a similar bootstrap argument yields that 
\begin{equation}
    \label{est:|na|^ga v^N}
    \xsc{\naabs{\ga}v^N}\les\nh{P_{>N}{v_0}}{\ga+s_c}+\nh{{v_0}}{\ga+s_c}\nh{P_{>N}{v_0}}{s_c}.
\end{equation}

Therefore, by \eqref{est:|na|^ga_0 v^N} and \eqref{est:|na|^ga v^N}, we establish
\begin{equation}
    \label{est:|na|^ga v^N-1}
    \xsc{\naabs{\ga}v^N}\les\nh{P_{>N}{v_0}}{\ga+s_c}+\chi_0\nh{{v_0}}{\ga+s_c}\nh{P_{>N}{v_0}}{s_c},
\end{equation}
where $\chi_0=0$ when $ -s_c \le \ga \le 0$, and $\chi_0=1$ otherwise.
In particular, we have for $\ga =-s_c$ and $v_0\in H^\al, \al \in [4,8]$,
\begin{equation}
    \label{est: v^N}
    \ntx{v^N}{\oo}{2}
    \les \nl{P_{>N}v_0}{2}
    \les N^{-\al}.
\end{equation}

\subsubsection{Estimate of $w^N$}
Suppose that $v_0\in H^\al_x, \al \in [4,8]$,
arguing analogously as the estimate of $v^N$, 
we aim to obtain the $L^\oo_tL^2_x$ estimate of solutions related to the high-frequency initial data $P_{>N}w_0$.

Similarly, denote $w_N$ to be the solution of the following equation:
\eq{\label{eq:f_N} 
    &2i\p_t w_N 
    -\De w_N  +\p_{tt} v_N +3 \left(\frac 18 |v_N|^4v_N+v^2_N  \ba{w_N } + 2|v_N|^2 w_N  \right)=0,\\
    & {w_{0}}_N
    =-\frac14 \left(P_{\le N}{v_0}\right)^3+\frac 18 \left(P_{\le N}\ba {v_0}\right)^3 +\frac14
    \left(\De P_{\le N}{v_0}-\De P_{\le N}\ba {v_0}\right)
    \\
    &\qquad \ \ 
     -\frac34\left(\left|P_{\le N}{v_0}\right|^2P_{\le N} {v_0}-\left|P_{\le N}{v_0}\right|^2P_{\le N}\ba {v_0}\right),}
and $w^N=w-w_N$ to be the solution of the following equation:
\eq{\label{eq:f^N} 
    &2i\p_t w^N  -\De w^N  +\p_{tt} v^N 
    +3\Big[
    \frac 18 \left(|v|^4v-|v_N|^4v_N\right)
    +\left(v^2 \ba w-v^2_N \ba  {w_N}\right)
    +\left(2|v|^2w-2|v_N|^2 w_N\right) 
    \Big]
    =0,\\
    & w_{0}^N =w_0-{w_0}_N.
    }
The corresponding Duhamel formula reads
\begin{align*}
   w^N(t)
   &=e^{-\frac i2 t\De }w^N_0
    +\frac{i}{2}\int_{0}^{t}e^{-\frac i2(t-s)\De}
    \Big\{
    \p_{ss} v^N   
   +3\Big[
   \frac 18 \left(|v|^4v-|v_N|^4v_N\right)        
   \\
   &\qquad
   +\left(v^2 \ba w-v^2_N \ba { w_N}\right)
   +\left(2|v|^2 w -2|v_N|^2w_N\right) 
   \Big]
   \Big\}(s)ds.
\end{align*}
By the definition of $w_0$ and ${w_0}_N$, we note that 
\[w_{0}^N=O \left(
{v_0}^2P_{>N}{v_0}+\De P_{>N}{v_0}
\right),\]
then we have for $  \ga \in [-s_c, 0]$,
\begin{align}
    \nh{\naabs{\ga}w_{0}^N}{s_c}
    &\les
    \nh{P_{> N}{v_0}}{\ga+s_c}\nh{{v_0}}{2+s_c}+\nh{P_{> N}{v_0}}{2+\ga+s_c}.
    \label{add est: w_0^N}
\end{align}
Note that 
\begin{align*}
    \p_{tt} v^N 
    &=O\left(
    \De^2 v^N 
    +\De \left(v^N v^2 \right)
    +v^2\De v^N
    +vv^N\De v
    +v^4v^N
    \right). 
\end{align*}
In light of this equality, Lemma \ref{strichartz estimate}, and Remark \ref{rem:chain rules},
 we have for $  \ga \in [-s_c, 0]$,
\begin{align*}
     & \xsc{\naabs{\ga}\int_{0}^{t}e^{-\frac i2 \left(t-s\right)\De}
          \p_{ss}v^N (s)ds} 
        \notag  \\
   &\quad
   \les
    \nth{\naabs{\ga+4} v^N}{1}{s_c}
    +\sc{\naabs{\ga+2}v^N}\sc v^2
    +\sc{v^N}\sc{\naabs{\ga+2}v}\sc v 
    \notag\\
    &\qquad
    +\sc{\naabs{\ga}v^N} \sc v^2 \ltx{v}{\oo}^2.
\end{align*}
This together with \eqref{est:|na|^ga v^N-1} and Lemma \ref{scattering theory}, yields that 
\begin{align}
    &\xsc{\naabs{\ga}\int_{0}^{t}e^{-\frac i2 \left(t-s\right)\De}
       \p_{ss}v^N (s)ds} 
       \notag\\
    &\quad \les
    \nh{P_{>N}{v_0}}{2+\ga+s_c}
    +\nh{{v_0}}{2+\ga+s_c}\nh{P_{>N}{v_0}}{s_c}
    +\nh{P_{>N}{v_0}}{\ga+s_c}\nh{{v_0}}{2+s_c}
    \notag \\
    &\qquad
    +T \left(
    \nh{P_{>N}{v_0}}{4+\ga+s_c} 
    +\nh{{v_0}}{4+\ga+s_c}\nh{P_{>N}{v_0}}{s_c}
    \right).
    \label{p_tt v^N}
\end{align}
By \eqref{p_tt v^N}, Lemma \ref{strichartz estimate}, and the Duhamel formula of $w^N$,
we obtain for $  \ga \in [-s_c, 0]$,
\begin{align*}
    \xsc{\naabs{\ga}w^N}
    & 
     \les
     \nh{\naabs{\ga}w_0^N}{s_c}
     + \eqref{p_tt v^N}
     +\zsc{\naabs{\ga}\left( |v|^4v-|v_N|^4v_N\right) }
     +\zsc{\naabs{\ga}\left( v^2\ba w-v^2_N\ba {w_N}\right) }
     \\
     &\quad
     +\zsc{\naabs{\ga}\left(|v|^2 w-|v_N|^2 w_N\right)}
     \\
     &\les
     \nh{\naabs{\ga}w_0^N}{s_c}
      + \eqref{p_tt v^N}
      +\sc{\naabs{\ga}v^N} \sc{v}^2 \ltx{v}{\oo}^2
      +\sc{\naabs{\ga}w^N}\sc v^2
      \\
      &\quad
      +\sc{v^N} \sc{v} \sc{\naabs{\ga} w}.
\end{align*}
Then due to \eqref{est:|na|^ga v^N-1}, \eqref{add est: w_0^N}, \eqref{p_tt v^N}, and Lemma \ref{lem:estimate of f}, we get
\begin{align*}
    \xsc{\naabs{\ga}w^N}
    &\les
        \nh{P_{>N}{v_0}}{2+\ga+s_c}
        +\nh{{v_0}}{2+\ga+s_c}\nh{P_{>N}{v_0}}{s_c}
        +\nh{P_{>N}{v_0}}{\ga+s_c}\nh{{v_0}}{2+s_c}
    \\
    &\quad
    +T \left(
    \nh{P_{>N}{v_0}}{4+\ga+s_c} 
    +\nh{{v_0}}{4+\ga+s_c}\nh{P_{>N}{v_0}}{s_c}
    \right)
    +\sc{v}^2\sc{\naabs{\ga}w^N}. 
\end{align*}
Since 
$ \sc{v}  < \oo,$ it follows from a standard bootstrap argument that for $  \ga \in [-s_c, 0]$,
\begin{align*}
   \xsc{\naabs{\ga}w^N} 
   &\les 
   \nh{P_{>N}{v_0}}{\ga+s_c}\nh{{v_0}}{2+s_c}
   +\nh{P_{>N}{v_0}}{2+\ga+s_c}
   +\nh{{v_0}}{2+\ga+s_c}\nh{P_{>N}{v_0}}{s_c}
   \\
   &\quad
   + T \left(
   \nh{P_{>N}{v_0}}{4+\ga+s_c} 
   +\nh{{v_0}}{4+\ga+s_c}\nh{P_{>N}{v_0}}{s_c}
   \right).
\end{align*}
This implies that for $\ga=-s_c$ and $ v_0\in H^\al_x, \al \in [4, 8]$,
\begin{align*}
    \norm{w^N}_{L^\oo_tL^2_x([0, T])}
    &\les N^{-\al}+N^{2-\al}+N^{s_c-\al}
           +T\left(N^{4-\al}+N^{s_c-\al}\right).
\end{align*}
Choosing $N>1$, we obtain
\begin{align}
    \label{est:f^N}
    \norm{w^N}_{L^\oo_tL^2_x([0, T])}
    &\les N^{2-\al}+ TN^{4-\al},
\end{align} 
where the implicit constant depends on $\norm{v_0}_{H^\al_x}$.
We thus have established the $L^\oo_tL^2_x $ estimate of $w^N$.

\subsubsection{Proof of Theorem \ref{main theorem}}
\label{proof of main theorem}
In this subsection, we finish the proof of Theorem \ref{main theorem}. 
For convenience, we denote
\begin{equation}
    \begin{aligned}
   {{\Phi_1}}_N:&=\ei{}v_N+c.c.,
    \\
   {{\Phi_2}}_N:&=\frac 18 \ei{3}{v_N}^3
            +\ei{}w_N+c.c..
\end{aligned}
\label{w1,w2}
\end{equation}
On account of \eqref{w1,w2} and \eqref{expansion of u}, we obtain
\begin{align*}
  \norm{r_1}_{L^\oo_tL^2_x([0, T])}
  &\les 
     \norm{u^\ve-{{\Phi_1}}_N-\ve^2{{\Phi_2}}_N}_{L^\oo_tL^2_x([0, T])}
     +\norm{\Phi_1-{{\Phi_1}}_N}_{L^\oo_tL^2_x([0, T])}
     \\
     &\quad
     +\ve^2\norm{\Phi_2- {{\Phi_2}}_N}_{L^\oo_tL^2_x([0, T])}
     \\
   &\les 
   \norm{u^\ve-{{\Phi_1}}_N-\ve^2{{\Phi_2}}_N}_{L^\oo_tL^2_x([0, T])} 
   +\norm{v^N}_{L^\oo_tL^2_x([0, T])} 
   \\
   &\quad
   +\ve^2\norm{v^3-v^3_N}_{L^\oo_tL^2_x([0, T])} 
   +\ve^2\norm{w^N}_{L^\oo_tL^2_x([0, T])}.
\end{align*}
Thanks to \eqref{est: v^N}, we have
\begin{align}
    \norm{v^3-v^3_N}_{L^\oo_tL^2_x([0, T])}
    \les
    \ntx{v^N}{\oo}{2}\ltx{v}{\oo}^2
    \les
    \norm{P_{>N}{v_0}}_{L_x^2}\nh{{v_0}}{2+s_c}\les
    N^{-\al}.
    \label{est:v^3-v_N^3}
\end{align}
By \eqref{est: v^N}, \eqref{est:f^N}, and \eqref{est:v^3-v_N^3},
 we obtain for small $\ve$ and $N>1$,
\begin{align}
    &\norm{v^N}_{L^\oo_tL^2_x([0, T])}    
    +\ve^2\norm{v^3-v^3_N}_{L^\oo_tL^2_x([0, T])} 
    +\ve^2\norm{w^N}_{L^\oo_tL^2_x([0, T])}
    \notag\\
    &\quad
    \les N^{-\al}
    \left(
    1+\ve^2N^{2}+\ve^2TN^{4}
    \right).
    \label{high frequency data}
\end{align}
This finishes the high-frequency part of the initial data of the estimate of $r_1$.

On the other hand, we apply
a similar argument to the proof of Proposition \ref{main prop} to estimate the ${L^\oo_tL^2_x}$ norm of $u^\ve-{{\Phi_1}}_N-\ve^2{{\Phi_2}}_N$.
Scaling back, we suppose that $T$ satisfies
\begin{align}
    \ve^{2}T\nh{{v_0}}{4+s_c}
    \le	C\left(\norm{v_0}_{H^{d-2}_x}\right) \de_0.
    \label{assumption of T(scaling back)}
\end{align}
Since $v_0\in H^\al_x, \al \in [4, 8]$,
it reduces to assume that for small $\ve, \de_0$, $T$ satisfies either
\begin{align}
    \label{de_0}
     \ve^2TN^{s_c}
    \le C\left(\norm{v_0}_{H^{4}_x}\right) \de_0,
    \quad \al \in [4, 4+s_c),
\end{align}
or
\begin{align}
    \label{new de_0}
    \ve^2T
    \le C\left(\norm{v_0}_{H^{4+s_c}_x}\right) \de_0,
    \quad
    \al \in [4+s_c, 8].
\end{align}
Then imitating the proof of Proposition \ref{main prop},
we obtain that under the assumption \eqref{de_0} or \eqref{new de_0},
\begin{align*}
    \norm{u^\ve-{{\Phi_1}}_N-\ve^2{{\Phi_2}}_N}_{L^\oo_tL^2_x([0, T])}
    &\les \ve^4
    \left(
        1+T \beta_\al (N)
        +T^2N^{8-\al}
    \right),
\end{align*}
where $\beta_\al (N)= N^{6-\al}$ for $\al \in [4,6)$, and $\beta_\al (N)= 1$ for $\al \in [6,8]$.
This estimate together with \eqref{high frequency data} yields that
\begin{align}
    \norm{r_1}_{L^\oo_tL^2_x([0, T])}
    &\les N^{-\al}
    \left(
    1+\ve^2N^{2}+\ve^2TN^{4}
    \right)
    +\ve^4
    \left(
    1+T \beta_\al (N)
    +T^2N^{8-\al}
    \right).
    \label{the mid-estimate of r_1}
\end{align}

Now we are in a position to decide the suitable frequency $N$ to get the sharp estimate of $\norm{r_1}_{L^\oo_tL^2_x([0, T])}$.
The proof is divided into the case $\al \in [4, 6)$ and the case $\al \in [6, 8].$

\textbf{Case A}: First, we consider $\al \in [4, 6)$ and divide it into the following subcases:
    
   $\bullet$ \textbf{Subcase $\text{A}_1$}: $1\ges \ve^2N^2, \ve^2TN^4$ and $1\ges TN^2$.
   
    In this case,  \eqref{the mid-estimate of r_1} becomes
    \begin{align*}
        \norm{r_1}_{L^\oo_tL^2_x([0, T])}
        \les \ve^4
            +N^{-\al}
            +\ve^4TN^{6-\al}.
    \end{align*} 
    Choosing $N=\left(\ve^4T\right)^{-\frac{1}{6}}$, we have that
    $T \sim \ve^2$, which further implies \eqref{de_0} or
    \eqref{new de_0}.
    Hence, noting that $T \sim \ve^2$,
    \[ \norm{r_1}_{L^\oo_tL^2_x([0, T])}
    \les \ve^4	
    +\left(\ve^4T \right)^{\frac{\al}{6}}
    \les  \ve^4. \]
    
    $\bullet$ \textbf{Subcase $\text{A}_2$}:  $1\ges \ve^2N^2, \ve^2TN^4$, and $ TN^2 \ges 1$.
    
    In this case, \eqref{the mid-estimate of r_1} reduces to
    \begin{align*}
        \norm{r_1}_{L^\oo_tL^2_x([0, T])}
        \les \ve^4
        +N^{-\al}
        +\ve^4T^2N^{8-\al}.
    \end{align*}     
    Choosing $N=\left(\ve^2T\right)^{-\frac{1}{4}}$, we have that $T \ges \ve^2$.    
    If 
    $\ve^2\les T\les \ve^{-2}\de_0$,
    then \eqref{de_0} or \eqref{new de_0} holds,
    and we have
    \begin{align*}
        \norm{r_1}_{L^\oo_tL^2_x([0, T])}
        \les \ve^4
        +\left(\ve^2T\right)^\frac{\al}{4}.
    \end{align*}  
    If $T\ges \ve^{-2}\de_0$,
    then arguing analogously as \eqref{est:r_1 case 2} shows that
    \begin{align*}
         \norm{r_1}_{L^\oo_tL^2_x([0, T])}
         &\les 1+\ve^2T
         \les  \left(\ve^2T\right)^\frac{\al}{4}.
    \end{align*}  
    Combining both estimates for $T$, we establish
        \begin{align*}
        \norm{r_1}_{L^\oo_tL^2_x([0, T])}
        \les \ve^4
        +\left(\ve^2T\right)^\frac{\al}{4}.
    \end{align*}
    
    $\bullet$ \textbf{Subcase $\text{A}_3$}: $\ve^2N^2 \ges 1, \ve^2TN^4$ and $ 1\ges TN^2$.
    
    In this case, \eqref{the mid-estimate of r_1} is simplified to
    \begin{align*}
        \norm{r_1}_{L^\oo_tL^2_x([0, T])}
        \les \ve^4
        +\ve^2N^{2-\al}
        +\ve^4TN^{6-\al}.
    \end{align*}
    Choosing $N=\left(\ve^2T\right)^{-\frac{1}{4}}$, we have that $T\les \ve^2$,
    which further implies \eqref{de_0} or \eqref{new de_0}.  
     Thus we have that
    \[
    \norm{r_1}_{L^\oo_tL^2_x([0, T])}
    \les \ve^4 
    +\ve^{{\frac{\al+2}{2}}}
    T^{\frac{\al-2}{4}}
    \les \ve^4. \]

    $\bullet$ \textbf{Subcase $\text{A}_4$}: $\ve^2N^2 \ges 1, \ve^2TN^4$, and $ TN^2 \ges 1$.
    
    In this case, \eqref{the mid-estimate of r_1} changes to
    \begin{align*}
        \norm{r_1}_{L^\oo_tL^2_x([0, T])}
        \les \ve^4
        +\ve^2N^{2-\al}
        +\ve^4T^2N^{8-\al}.
    \end{align*}  
    Choosing $N=\left(\ve T\right)^{-\frac{1}{3}}$, we have that $T\sim \ve^2$,
    which further implies \eqref{de_0} or
    \eqref{new de_0}.
    It infers that
    \[ \norm{r_1}_{L^\oo_tL^2_x([0, T])}
    \les \ve^4	
    +\ve^2\left(\ve T \right)^{\frac{\al-2}{3}}
    \les \ve^4. \]
    
    $\bullet$ \textbf{Subcase $\text{A}_5$}: $ \ve^2TN^4\ges 1,\ve^2N^2$ and $ 1\ges TN^2 $.
    
    It follows from \eqref{the mid-estimate of r_1} that
    \begin{align*}
        \norm{r_1}_{L^\oo_tL^2_x([0, T])}
        \les \ve^4
        +\ve^2TN^{4-\al}+\ve^4TN^{6-\al}.
    \end{align*}  
    Choosing $N=\ve^{-1}$, we have that $T\sim \ve^2$,
    which further implies \eqref{de_0} or
    \eqref{new de_0}.
    Therefore, we deduce that
    \[ \norm{r_1}_{L^\oo_tL^2_x([0, T])}
    \les \ve^4	
    +\ve^{\al-2}T
    \les \ve^4. \]
    
    $\bullet$ \textbf{Subcase $\text{A}_6$}: $ \ve^2TN^4\ges 1,\ve^2N^2$, and $ TN^2 \ges 1$.
    
    We obtain from \eqref{the mid-estimate of r_1} that, 
    \begin{align*}
        \norm{r_1}_{L^\oo_tL^2_x([0, T])}
        \les \ve^4
        +\ve^2TN^{4-\al}
        +\ve^4T^2N^{8-\al}.
    \end{align*}     
    Choosing $N=\left(\ve^2T\right)^{-\frac{1}{4}}$, we have that $T \ges \ve^2$. The above estimate reduces to 
    \begin{align*}
        \norm{r_1}_{L^\oo_tL^2_x([0, T])}
        \les \ve^4
        +\left(\ve^2T\right)^\frac{\al}{4}.
    \end{align*}
     Proceeding similarly as \textbf{Subcase $\text{A}_2$},
     it turns out that
    \begin{align*}
        \norm{r_1}_{L^\oo_tL^2_x([0, T])}
        \les \ve^4
        +\left(\ve^2T\right)^\frac{\al}{4}.
    \end{align*}
    
\textbf{Case B}: Now, we consider the case $\al \in [6, 8]$ and divide it into the following cases:
The proof is similar to the case $\al \in [4, 6)$.

     $\bullet$ \textbf{Subcase $\text{B}_1$}: $1\ges \ve^2N^2, \ve^2TN^4$.
    
    In this case, \eqref{the mid-estimate of r_1} becomes
    \begin{align*}
        \norm{r_1}_{L^\oo_tL^2_x([0, T])}
        \les \ve^4(1+T)
        +N^{-\al}
        +\ve^4T^2N^{8-\al}.
    \end{align*}     
    Choosing $N=\left(\ve^2T\right)^{-\frac{1}{4}}$, we have that $T \ges \ve^2$.    
    If 
    $\ve^2\les T\les \ve^{-2}\de_0$,
    then \eqref{new de_0} holds 
    and we obtain that
    \begin{align*}
       \norm{r_1}_{L^\oo_tL^2_x([0, T])}
       &\les \ve^4(1+T)
       +\left(\ve^2T\right)^\frac{\al}{4}
       \\
       &\les \ve^4
       +\left(\ve^2T\right)^\frac{\al}{4}. 
    \end{align*}
    If $T\ges \ve^{-2}\de_0$,
    then a similar argument as \eqref{est:r_1 case 2} 
    yields that 
    \begin{align*}
        \norm{r_1}_{L^\oo_tL^2_x([0, T])}
        &\les 1+\ve^2T
        \les  \left(\ve^2T\right)^\frac{\al}{4}.
    \end{align*}  
     Gathering both results for $T$, we get that
    \begin{align*}
        \norm{r_1}_{L^\oo_tL^2_x([0, T])}
        \les \ve^4
        +\left(\ve^2T\right)^\frac{\al}{4}.
    \end{align*}
    
     $\bullet$ \textbf{Subcase $\text{B}_2$}: $\ve^2N^2 \ges 1, \ve^2TN^4$.
    
    In this case, \eqref{the mid-estimate of r_1} becomes
    \begin{align*}
        \norm{r_1}_{L^\oo_tL^2_x([0, T])}
        \les \ve^4(1+T)
        +\ve^2N^{2-\al}
        +\ve^4T^2N^{8-\al}.
    \end{align*}  
    Choosing $N=\left(\ve T\right)^{-\frac{1}{3}}$, we have that $T\les  \ve^2$,
    which further implies \eqref{new de_0}.
    As a result, we have that
    \[ \norm{r_1}_{L^\oo_tL^2_x([0, T])}
    \les \ve^4	(1+T)
    +\ve^2\left(\ve T \right)^{\frac{\al-2}{3}}
    \les  \ve^4. \]

    $\bullet$ \textbf{Subcase $\text{B}_3$}: $ \ve^2TN^4\ges 1,\ve^2N^2$.
    
    In this case, \eqref{the mid-estimate of r_1} becomes
    \begin{align*}
        \norm{r_1}_{L^\oo_tL^2_x([0, T])}
        \les \ve^4(1+T)
        +\ve^2TN^{4-\al}
        +\ve^4T^2N^{8-\al}.
    \end{align*}     
    Choosing $N=\left(\ve^2T\right)^{-\frac{1}{4}}$, we have that $T \ges \ve^2$.
    The subsequent proof process is similar to 
    \textbf{Subcase $\text{B}_1$},
    thus we omit the proof.
    As a consequence, we obtain that
    \begin{align*}
        \norm{r_1}_{L^\oo_tL^2_x([0, T])}
        \les \ve^4
        +\left(\ve^2T\right)^\frac{\al}{4}.
    \end{align*}
        
Assembling those results above in \textbf{Case A} and \textbf{Case B}, we get
 for all $T \ge 0 $ with $\al \in [4,8] $,
\begin{align*}
    \norm{r_1}_{L^\oo_tL^2_x([0, T])}
    \les \ve^4
    +\left(\ve^2T\right)^\frac{\al}{4}.
\end{align*}
This finishes the proof of the theorem.

\section{Optimal convergence rate}
\label{Optimal convergence rate}
In this section, we aim to get the optimal convergence rate of $r_1$.
We shall identify the accurate formula which dominates the estimate of $\phi$ first.
Then we deduce the lower bound of $\phi$ from this formula and thus obtain the lower bound control of $r_1$.
As a result, we get the optimal convergence rate of $r_1$ after choosing suitable initial data $v_0$.

Let 
\[ v_0(x)=\de_0 g(x), \]
where $g(x)\in \mathscr S (\R^d)$ is a real-valued function, and $\de_0$ is a small constant which will be chosen later. 
Recall the Duhamel formula of $\phi(\ve^{-2}t)$,
\begin{align*}
    \phi(\ve^{-2}t)=e^{i\ve^{-2}t\jd{}}{\phi_0} 
    -\int_{0}^{\ve^{-2}t}e^{i(\ve^{-2}t-s)\jd{}}\jd{-1}({G_1}+{G_2}+{G_3}+{G_R})(s) ds,
\end{align*}                         
we need to estimate each part of the Duhamel formula of $\phi(\ve^{-2}t)$ with $v_0(x)=\de_0 g(x)$ as the proof of Lemma \ref{lem:estimate of phi}.
Then we find out which part in the Duhamel formula dominates the estimate of $\phi(\ve^{-2}t)$.
Suppose that for $\de_0$ suitably small,
\begin{align*}
   t\in I:= \big\{t\in \R_+: 1 \le t \le \ve^{-2}\de_0 \big\}, 
\end{align*}
then $t$ satisfies the condition after scaling in Lemma \ref{lem:estimate of phi}.
As a result, we can reproduce the proof process of Lemma \ref{lem:estimate of phi} to obtain the estimates for each part $G_1$--$G_R$.

We are going to give the estimates of each part for $G_1$--$G_R$.
By Lemma \ref{lem:estimate of f},
we obtain that for $t\in I, \ga \ge -s_c,$
\begin{equation}
    \begin{aligned}
   \norm{\naabs{\ga}w(t)}_{X_{0}([0, t])}
   &\les \de_0\left(\nh{g}{2+\ga}
                    +t\nh{g}{4+\ga} \right),
   \\
   \norm{\naabs{\ga}(\p_t w)(t)}_{X_{0}([0, t])}
   &\les \de_0\left(\nh{g}{4+\ga}
            +t\nh{g}{6+\ga} \right).
\end{aligned}
 \label{est: f_de_0}
\end{equation}
Scaling back, \eqref{est: f_de_0} implies that for $t\in I, \ga \ge -s_c,$
\begin{equation}
    \begin{aligned}
    \norm{\naabs{\ga}\ti w(\ve^{-2}t)}_{X_{0}([0, t])}
    &\les \ve^{2+\ga-s_c}\de_0
    \left(\nh{g}{2+\ga}
    +t\nh{g}{4+\ga} \right),
    \\
    \norm{\naabs{\ga}\p_t\ti w(\ve^{-2}t)}_{X_{0}([0, t])}
    &\les \ve^{4+\ga-s_c}\de_0
    \left(\nh{g}{4+\ga}
    +t\nh{g}{6+\ga} \right).
    \end{aligned}
\label{est: ti f_de_0}
\end{equation}
By \eqref{est: ti f_de_0} and \eqref{est: {G_1}} with $\ga=-s_c$, we have for $t\in I,$
\begin{align}
    \norm{\int_{0}^{\ve^{-2}t}e^{i(\ve^{-2}t-s)\jd{}}\jd{-1}G_1(s)ds}
    _{Y_0([0, t])}
    &\les \ve^{4-s_c}\de_0  \left(\de_0^2\nh{g}{4}+t\nh{g}{6}+t^2\nh{g}{8}\right).
    \label{est: G_1_de}
\end{align}
Similarly, by \eqref{est: {G_2}} and \eqref{est:{G_3}},
we obtain for $t\in I,$
\begin{align}
    \norm{\int_{0}^{\ve^{-2}t}e^{i(\ve^{-2}t-s)\jd{}}\jd{-1}G_2(s)ds}_{Y_0([0, t])}
    &\les \ve^{{4-s_c}}\de_0^3\left(\nh{g}{4}+t\nh{g}{6}
    \right),
    \label{est: G_2_de}
\end{align}
and
\begin{align}
    &\norm{\int_{0}^{\ve^{-2}t}e^{i(\ve^{-2}t-s)\jd{}}\jd{-1}G_3(s)ds}_{Y_0([0, t])}
\les \ve^{4-s_c}\de_0^3\left(
        \nh{g}{4}
        +t\nh{g}{6}
        +t^2\nh{g}{8}
    \right).
    \label{est: G_3_de}
\end{align}
By \eqref{est:phi0} and \eqref{est: f_de_0}, we have
\begin{align}
    \nl{\jd{\frac12}{\phi_0}}{2}   
    &\les \ve^{{4-s_c}}\left(\nl{\p_t( v^{3})(0)}{2}
    +\nl{\p_t w(0)}{2}\right)
 \les 
 \ve^{{4-s_c}}\de_0 \nh{g}{4}.
    \label{est:phi0_de}
\end{align}
Combining those estimates \eqref{est:{G_R}}, \eqref{est: G_1_de}--\eqref{est:phi0_de}
and arguing similarly as Section \ref{the regular case}(note that $t\in I$ implies that \eqref{assumption of T(scaling back)} hold),  
we find that 
\begin{align}
\nl{R_1(\ve^{-2}t)}{2}
    &\les \nl{\phi(\ve^{-2}t)}{2}
\les 
    \ve^{4-s_c}\de_0\left(
    \nh{g}{4}
    +t\nh{g}{6}
    +t^2\nh{g}{8}
    \right).
    \label{est:R_1_de}
\end{align}
This in turn implies that 
\begin{align}
    &\norm{\int_{0}^{\ve^{-2}t}e^{i(\ve^{-2}t-s)\jd{}}\jd{-1}G_R(s)ds}_{Y_0([0, t])}
    \les  \ve^{4-s_c}\de_0^3\left(
    \nh{g}{4}
    +t\nh{g}{6}
    +t^2\nh{g}{8}
    \right).
    \label{est: G_R_de}
\end{align}
Comparing those estimates 
\eqref{est: G_1_de}--\eqref{est: G_R_de}, 
we infer that the dominant term of $\phi(\ve^{-2}t)$ 
is involved in the integral of $G_1$.
Consequently, to get the final sharp convergence rate of $r_1$,
a more detailed analysis about \eqref{est: G_1_de} is needed.

\subsection{The approximation formula}
In this subsection, our goal is to establish the approximated formula for the $G_1$ term.
By the definition \eqref{G_1} of $G_1$, we have
\begin{align}
   &\int_{0}^{\ve^{-2}t}e^{i(\ve^{-2}t-s)\jd{}}\jd{-1}G_1(s)ds
   \notag\\
   &\qquad
   = \int_{0}^{\ve^{-2}t}e^{i(\ve^{-2}t-s)\jd{}}\jd{-1}
    \left(
     e^{it}\p_{tt}\ti w +e^{-it}\p_{tt} \ba {\ti w}
    \right)(s)ds
    +\mathcal N_1(h),
    \label{G_1_de_0}
\end{align}
where 
\begin{align*}
   \mathcal N_1(h)
   =\frac18
   \int_{0}^{\ve^{-2}t}e^{i(\ve^{-2}t-s)\jd{}}\jd{-1}
   \left[
        e^{3it}\p_{tt} \left( h^3\right)
        + e^{-3it}\p_{tt} \left(\ba h^3\right)
   \right] (s)ds. 
\end{align*}
It follows from \eqref{est: {G_1}1} that 
\begin{align}
    \norm{\mathcal N_1(h)}_{L^\oo_tL^2_x([0, t])}
    \les \ve^{4-s_c} \de_0^3\nh{g}{4}.
    \label{est:N_1}
\end{align}
Since $\ti w $ also verifies equation \eqref{eq:f} with initial data
$\ti w_0 =\ve^2 S_\ve w_0$, invoking \eqref{p_tt ti f} and \eqref{p_ttth},
 we obtain
\begin{align*}
   \p_{tt} \ti w
   &=
   -\frac 14 \De^2 \ti w
   -\frac 18 \De^3 h
   + \mathcal O(h, \ti w ).
\end{align*}
where
\begin{align*}
\mathcal O (h, \ti w )
&=O \bigg(
\De  \left(  h^2 \ti w\right)   
+h^4\p_t h
+\p_t\left( h^2\ti  w\right)
+\De\left(h^2\De h\right)
+h^2\De (h^3) +\De( h^5)+h^7
\bigg). 
\end{align*}
Inserting the above formula of $\p_{tt}\ti w$ into \eqref{G_1_de_0}, then
\begin{align}
   \eqref{G_1_de_0}
   &= \int_{0}^{\ve^{-2}t}e^{i(\ve^{-2}t-s)\jd{}}\jd{-1}
   \left[
   e^{is}\left(
   -\frac 14 \De^2 \ti w
   -\frac 18 \De^3 h
   \right)   (s)
   +c.c.
   \right]   (s)ds
\notag
   \\
   &\quad
   +\mathcal N_1(h)
   +\mathcal N_2(h, \ti w ),
   \label{leading term}
\end{align}
where
\begin{align*}
    \mathcal N_2(h, \ti w ) =\int_{0}^{\ve^{-2}t}e^{i(\ve^{-2}t-s)\jd{}}\jd{-1} 
        \mathcal O (h, \ti w )
    (s)ds.
\end{align*}
It follows from \eqref{4.20+4.21} and \eqref{est:4.19} that 
\begin{align}
   \norm{\mathcal N_2(h, \ti w)}_{L^\oo_tL^2_x([0, t])}
   &\les 
    \ve^{4-s_c}\de_0^3\left(\nh{g}{4}
        +t\nh{g}{6}\right).
    \label{est:N_2}
\end{align}
Therefore, combining \eqref{G_1_de_0} and \eqref{leading term},
we need the approximated formula of
$\ti w $.

\subsubsection{The approximation formula of $\ti w $}
Recall the Duhamel formula of $\ti w$:
\begin{align*}
      \ti w(t)
   &=e^{-\frac{i}{2}t\De}\ti w_0
           +\frac{i}{2}\int_0^{t} e^{-\frac{i}{2}(t-s)\De}
           \left[
                \p_{ss}h
                +3\left(\frac18 |h|^4 h
                        +h^2\ba {\ti  w}
                        +2|h|^2 \ti w
                  \right)
           \right](s)ds.
\end{align*}
Since $g(x)$ is real-valued, then \eqref{eq:f} implies that $\ti w_0=-\frac 18 h_0^3$. We rewrite the Duhamel formula of $\ti w $ as follows:  
\begin{align}
    \ti w(t)
    = 
    -\frac{i}{8}\int_0^{t} e^{-\frac{i}{2}(t-s)\De}
    \De^2 h(s)ds
    +\mathcal N_3(h, \ti w),
    \label{Duhamel of ti f_de_0 }
\end{align}
where 
\begin{align*}
     \mathcal N_3(h, \ti w)
    &= -\frac 18 e^{-\frac{i}{2}t\De} h_0^3
    +\frac{i}{2}\int_0^{t} e^{-\frac{i}{2}(t-s)\De}
    O\left( h^2\De h+ \De  \left(h^3\right)+h^5
        +h^2\ti w \right)   (s)ds, 
\end{align*}
and 
\begin{align}
    \norm{\mathcal N_3(h, \ti w)}_{L^\oo_tL^2_x([0, t])}
    &\les \ve^{2-s_c}\de_0^3
              \left(
                \nh{g}{2}
                +\ve^2t\nh{g}{4}
            \right).
    \label{est:N_3}        
\end{align}
To approximate further \eqref{Duhamel of ti f_de_0 }, 
using the Duhamel formula of $h(t)$,
\begin{align}
    h(t) =e^{-\frac i 2 t\De }h_0
    +\frac {3i}2\int_{0}^{t}  e^{-\frac i2(t-s)\De }
    \big(|h|^2h\big)(s)ds,
    \label{Duhamel of h}
\end{align}
 we have the approximation formula of $\ti w $,
 \begin{align}
    \ti w(t)  
    &=  
    -\frac{i}{8} e^{-\frac{i}{2}t\De}
    t\De^2 h_0
    +\mathcal N_3(h, \ti w)
    +\mathcal N_4(h).
    \label{Duhamel of ti f:expansion}
\end{align}
where 
\begin{align*}
    \mathcal N_4(h)
    = \frac{3}{16}\int_0^{t} e^{-\frac{i}{2}(t-s)\De}
    \left(
    \int_{0}^{s}  e^{-\frac i2(s-s_1)\De }
    \De^2 \big(|h|^2h\big)(s_1)ds_1
    \right)ds,
\end{align*}
and 
\begin{align}
    \norm{\mathcal N_4(h)}_{L^\oo_tL^2_x([0, t])}
    &\les \ve^{4-s_c}\de_0^3 t
    \nh{g}{4}.
    \label{est:N_4}
\end{align}

\subsubsection{The approximation formula of the $G_1$ part}

Inserting \eqref{Duhamel of h} and \eqref{Duhamel of ti f:expansion} to \eqref{leading term},
we have the following approximation formula: 
\begin{align}
   \int_{0}^{\ve^{-2}t}e^{i(\ve^{-2}t-s)\jd{}}\jd{-1}G_1(s)ds
  &=:
    I_1(t)+I_2(t)
    +\mathcal N_1(h)
    +\mathcal N_2(h, \ti w )
    +\mathcal N(h, \ti w ),
    \label{leading term expansion}
\end{align}
where
\begin{align*}
   I_1(t)
   &=e^{i\ve^{-2}t\jd{}}
   \int_{0}^{\ve^{-2}t}e^{is\left(-\jd{}+1-\frac \De 2\right)}\jd{-1}
   \left(
   -\frac{1}{8}\De^3 h_0
   +\frac i{32} s\De^4 h_0
   \right) ds,
   \\
   I_2(t)
   &=e^{i\ve^{-2}t\jd{}}
   \int_{0}^{\ve^{-2}t}e^{-is\left(\jd{}+1-\frac \De 2\right)}\jd{-1}
   \left(
   -\frac{1}{8}\De^3  h_0
   -\frac i{32} s\De^4 h_0
   \right)
   ds,
   \\ 
   \mathcal N(h, \ti w )
   &= -\frac 14
    \int_{0}^{\ve^{-2}t}e^{i(\ve^{-2}t-s)\jd{}}\jd{-1}
   \big[
    e^{is}\De^2 \left( \mathcal N_3(h, \ti w )
            +\mathcal N_4(h)\right)
    +c.c.
    \big](s)ds
    \\
    &\quad
    -\frac{3}{16}
    \int_{0}^{\ve^{-2}t}e^{i(\ve^{-2}t-s)\jd{}}\jd{-1}
    \bigg[
    i e^{is}\De^3 \int_{0}^{s}  e^{-\frac i2(s-s_1)\De }
    \big(|h|^2h\big)(s_1)ds_1
    +c.c
     \bigg]
    ds.
    \end{align*}
Thanks to \eqref{est:N_3} and \eqref{est:N_4}, we get that 
\begin{align}
   \norm{\mathcal N(h, \ti w )}_{L^\oo_tL^2_x([0, t])}
   &\les 
    {\ve^{-2}t}
    \left(
        \norm{\De^2\mathcal N_3(h, \ti w)}_{L^\oo_tL^2_x([0, \ve^{-2}t])}
        +\norm{\De^2\mathcal N_4(h)}_{L^\oo_tL^2_x([0, \ve^{-2}t])}
        \right)
        \notag\\
        &\quad
    +\ve^{-2}t
     \norm{\De^3 \int_{0}^{s}  e^{\mp\frac i2(s-s_1)\De }
         \big(|h|^2 \ba h\big)(s_1)ds_1}_{L^\oo_sL^2_x([0, \ve^{-2}t])}        
  \notag\\
  &\les 
     \ve^{4-s_c}\de_0^3 \left(
        t\nh{g}{6}
       +t^2\nh{g}{8}
     \right).
     \label{est:N}
\end{align}
Gathering \eqref{est:N_1}, \eqref{est:N_2}, and \eqref{est:N}, we have
\begin{align}
     \norm{\mathcal N_1(h)
    +\mathcal N_2(h, \ti w )
    +\mathcal N(h, \ti w )}_{L^\oo_tL^2_x([0, t])}
     \les 
     \ve^{4-s_c}\de_0^3 \left(
     \nh{g}{4}
       + t\nh{g}{6}
       +t^2\nh{g}{8}
     \right).
     \label{G_1:expansion}
\end{align}

\subsection{Lower bound control of $r_1(t)$}
Now that the approximation formula \eqref{G_1:expansion} is established,
 we can give a lower bound control of $r_1(t)$ as follows.
\begin{lem}\label{lem:est of phi_de}
    Let $g\in \mathscr S\left(\R^d\right)$ be a real-valued function,
    then there exist constants $C_0, C_1, C_2, C$ such that
    for $t\in I$,     
    we have
    \begin{align*}
        \nl{r_1(t))}{2}
        &\ge
        C_1 \ve^{4}\de_0t^2 \nh{g}{8}
        + C_2 \ve^{6}\de_0t^3 \nh{g}{12}
        -C_0 \ve^{4}\de_0t \nh{g}{6}
        \notag\\
        &\quad
        -C\ve^{4}\de_0 
        -C\ve^{4}\de_0^3
        \left(
        1
        +t\nh{g}{6}
        +t^2\nh{g}{8}
        \right)
        \notag\\
        &\quad 
        - C\ve^{6}\de_0
        \left(
        \nh{g}{6}
        +t\nh{g}{8}
        +t^2\nh{g}{10}
        \right)
        \notag\\
        &\quad 
        -C\ve^{8}\de_0
        \left(
        \nh{g}{8}
        +t^2\nh{g}{12}
        +t^3\nh{g}{14}
        \right),
    \end{align*}
    where positive constants $C_0, C_1, C_2, C$ depend on $\nh{g}{4}$ and are independent of $\ve, \de_0, t$. 
\end{lem}
\begin{proof}
    We divide the proof into two steps.
     
Step 1: our task is to establish the estimates of $I_1, I_2$
and the lower bound control of $\phi(\ve^{-2}t)$.

$\bullet$ Estimation of $I_1(t)$.
Denote the operator 
\[ T_t 
=\frac{e^{it\left(-\jd{}+1-\frac \De 2\right)}-1-\frac18 it\De^2}
    {t\De^3}. \]
By the Mihlin-H\"ormander Multiplier Theorem, the operator
$T_t: L^p_x\to L^p_x, p\in (1, \oo)$ is  bounded uniformly in $t$.
Moreover, 
\begin{align*}
   \nl{T_t \varphi(x)}{p} 
   \les \nl{\varphi(x)}{p},
   \quad 
   p\in (1, \oo),
\end{align*}
where the implicit constant is independent of $t$ and $\varphi(x)$.
Due to the definition of $T_t$, we write 
\begin{align*}
   e^{it\left(-\jd{}+1-\frac \De 2\right)}
    =1+\frac18 it\De^2 
    +t\De^3 T_t.
\end{align*}
In view of this formula, we obtain that 
\begin{subequations}
    \label{I_1}
    \begin{align}
        I_1(t) 
        &=
        \ve^4\de_0 t  
        e^{i\ve^{-2}t\jd{}}
        \left(
        -\frac{1}{8}S_\ve (\De^3 g)
        +\frac i{64} tS_\ve (\De^4 g)
        \right)
        -\frac{1}{768}\ve^6\de_0t^3
        e^{i\ve^{-2}t\jd{}}
         S_\ve  \left( \De^6 g \right)
        \label{I_1-1}
        \\
        &\quad
        -\frac{i}{128}\ve^6\de_0t^2
        e^{i\ve^{-2}t\jd{}}
           \jd{-1} S_\ve \left( \De^5 g\right) 
        -\frac{1}{768}\ve^6\de_0t^3
        e^{i\ve^{-2}t\jd{}}
        \big(\jd{-1}-1\big) S_\ve  \left( \De^6 g \right)     
          \label{I_1-2}\\
        &\quad 
        +
        \ve^4\de_0 t  
        e^{i\ve^{-2}t\jd{}}
        \left(\jd{-1}-1\right)
        \left(
        -\frac{1}{8}S_\ve (\De^3 g)
        +\frac i{64} tS_\ve (\De^4 g)
        \right)
        \label{I_1-error-1}\\
        &\quad
        +e^{i\ve^{-2}t\jd{}}
        \int_{0}^{\ve^{-2}t}
        s\De^3T_s \jd{-1}
        \left(
        -\frac{1}{8}\De^3 h_0
        +\frac i{32} s\De^4 h_0
        \right) ds.
        \label{I_1-error-2}
    \end{align}
\end{subequations}  
First, for \eqref{I_1-1}, since $g$ is real-valued,
we have for $t\in I,$
\begin{align*}
    \nl{\eqref{I_1-1}}{2}^2
    &= \nl{
        -\frac 18 \ve^4 \de_0 t S_\ve (\De ^3 g)
        -\frac 1{768} \ve^6 \de_0 t^3 S_\ve  (\De^6 g)
    }{2}^2
        +\left( \frac{1}{64} \ve^{4-s_c}\de_0t^2 \nh{g}{8}\right)^2.
\end{align*}
This implies that
\begin{align}
    \nl{\eqref{I_1-1}}{2}
    &\ge 
       \frac 12 
            \nl{
                -\frac 18 \ve^4 \de_0 t S_\ve (\De ^3 g)
                -\frac 1{768} \ve^6 \de_0 t^3 S_\ve  (\De^6 g)
            }{2} 
       + \frac{1}{128} \ve^{4-s_c}\de_0t^2 \nh{g}{8}
      \notag\\
    &\ge
      -\frac1{16} \ve^{4-s_c}\de_0t \nh{g}{6} 
    +  \frac{1}{128} \ve^{4-s_c}\de_0t^2 \nh{g}{8}
    + \frac{1}{1536}\ve^{6-s_c}\de_0t^3\nh{g}{12}.  
           \label{est:I_1-1}
\end{align}
Note that
\[ \jd{-1}-1 =\frac{\De }{\jd{}\left(1+\jd{}\right)}, \]    
then we obtain for $t\in I,$
\begin{align}
    \nl{\eqref{I_1-2}}{2}
    &\les  
    \ve^{6-s_c}\de_0 t^2\nh{g}{10}
    +\ve^{8-s_c}\de_0 t^3 \nh{g}{14},
    \label{est:I_1-2}
\end{align}
and
\begin{align}
    \nl{\eqref{I_1-error-1}}{2}
   &\les 
        \ve^{6-s_c}\de_0t \nh{g}{8}
        +\ve^{6-s_c}\de_0 t^2\nh{g}{10}.       
        \label{est:I_1-error-1}
\end{align}
By the $L^p_x$ uniform boundedness of $T_t$,
we obtain that
\begin{align}
   \nl{\eqref{I_1-error-2}}{2}
   &\les 
       \int_{0}^{\ve^{-2}t}
       s  \nl{\De^3T_s \jd{-1}\De^3 h_0}{2}ds
      + \int_{0}^{\ve^{-2}t}
        s^2
      \nl{\De^3T_s \jd{-1} \De^4 h_0 }{2}ds
     \notag \\
  &\les 
       \ve^{8-s_c}\de_0 
       \left(t^2
       \nh{g}{12}
       +t^3\nh{g}{14}
       \right).
       \label{est:I_1-error-2}
\end{align} 
By \eqref{I_1}--\eqref{est:I_1-error-2}, we have
\begin{align}
   \nl{I_1}{2}
   &\ge  -\frac1{16} \ve^{4-s_c}\de_0t \nh{g}{6} 
       +  \frac{1}{128} \ve^{4-s_c}\de_0t^2 \nh{g}{8}
       + \frac{1}{1536}\ve^{6-s_c}\de_0t^3
       \nh{g}{12}
       \notag\\
       &\quad
       - C\ve^{6-s_c}\de_0
       \left(
       t\nh{g}{8}
       +t^2\nh{g}{10}
       \right)
       -C\ve^{8-s_c}\de_0
       \left(
       t^2\nh{g}{12}
       +t^3\nh{g}{14}
       \right).
             \label{est:I_1}  
\end{align}
This finishes the estimate of $I_1(t)$.

$\bullet$ Estimation of $I_2(t)$.
Integration by parts gives that
\begin{align*}
   I_2(t)
   &= e^{i\ve^{-2}t\jd{}}
   \int_{0}^{\ve^{-2}t}
   \frac{\jd{-1}}{-i\left(\jd{}+1-\frac \De2\right)}
    \left(
    -\frac{1}{8}\De^3  h_0
    -\frac i{32} s\De^4 h_0
    \right)
    de^{-is\left(\jd{}+1-\frac \De 2\right)}
    \\
   &=
   \frac{e^{-i\ve^{-2}t (1-\frac{\De}{2})}
       -e^{i\ve^{-2}t\jd{}}}
       {-i\left( \jd{}+1-\frac \De 2\right)}
       \jd{-1}
     \bigg( -\frac{1}{8}\De^3  h_0
            -\frac{1}{32(\jd{}+1-\frac \De 2)}\De^4 h_0
        \bigg)
        \\
     &\quad
     +\frac{e^{-i\ve^{-2}t (1-\frac{\De}{2})}}{32(\jd{}+1-\frac \De 2)} \jd{-1}    \ve^{-2}t\De^4 h_0.
\end{align*}
This implies that 
\begin{align}
   \nl{I_2(t)}{2}
   &\les  
    \ve^{6-s_c}\de_0 \nh{g}{6}
    +\ve^{6-s_c}\de_0 t \nh{g}{8}
    +\ve^{8-s_c}\de_0 \nh{g}{8},
    \label{est:I_2}  
\end{align} 
which finishes the estimate of $I_2(t)$.

$\bullet$ The lower bound control of $\phi(\ve^{-2}t)$.
Gathering those estimates \eqref{est: G_2_de},
\eqref{est: G_3_de}, \eqref{est:phi0_de}, and \eqref{est: G_R_de}, we have for $t\in I$,
\begin{align*}
    &\nl{\jd{\frac12}{\phi_0}}{2} 
    +\norm{ \int_{0}^{\ve^{-2}t}e^{i(\ve^{-2}t-s)\jd{}}\jd{-1}
       ( G_2+G_3+G_R)(s)ds }_{L^\oo_tL^2_x([0, t])}
       \\
      &\quad
      \le
   C\ve^{4-s_c}\de_0 \nh{g}{4}
   +C\ve^{4-s_c}\de_0^3
   \left(
   \nh{g}{4}
   +t\nh{g}{6}
   +t^2\nh{g}{8}
   \right). 
\end{align*}
Substituting the above estimate, \eqref{leading term expansion}, \eqref{G_1:expansion}, \eqref{est:I_1}, and \eqref{est:I_2} into the Duhamel formula of $\phi(\ve^{-2t})$,
we have for any $t\in I$,
\begin{align}
   \nl{\phi(\ve^{-2}t)}{2}
   &\ge
   \frac{1}{128} \ve^{4-s_c}\de_0t^2 \nh{g}{8}
   + \frac{1}{1536}\ve^{6-s_c}\de_0t^3\nh{g}{12} 
   -\frac1{16} \ve^{4-s_c}\de_0t \nh{g}{6}
   \notag\\
   &\quad
   -C\ve^{4-s_c}\de_0 \nh{g}{4}
   -C\ve^{4-s_c}\de_0^3
   \left(
   \nh{g}{4}
   +t\nh{g}{6}
   +t^2\nh{g}{8}
   \right)
   \notag\\
   &\quad 
   - C\ve^{6-s_c}\de_0
   \left(
   \nh{g}{6}
   +t\nh{g}{8}
   +t^2\nh{g}{10}
   \right)
   \notag\\
   &\quad 
   -C\ve^{8-s_c}\de_0
   \left(
   \nh{g}{8}
   +t^2\nh{g}{12}
   +t^3\nh{g}{14}
   \right).
   \label{phiphi} 
\end{align}
Thus we accomplish the first step.

Step 2: we need to show that $R_1(\ve^{-2}t), t\in I$ satisfies a similar lower $L^2_x$-bound as in \eqref{phiphi}.
For convenience, we denote the right-hand side of \eqref{phiphi} by $A(t)$. 

Firstly, we claim that if for any $t\in I,$
\begin{align}
    \nl{\phi(\ve^{-2}t)}{2}  \ge A(t),
     \label{phi-0}
\end{align} 
then there exists some $s\in [t, t+\ve^2]\subset I$ such that
\begin{align}
    \nl{R_1(\ve^{-2}s) }{2} \ge \frac 18 A(s).
    \label{phi-0-0}
\end{align} 
If $A(s)\le 0$, we are done.
For $A(s) > 0$,
we argue by contradiction.
Suppose that for any $s\in [t, t+\ve^2],$
we have
\begin{align}
    \nl{R_1(\ve^{-2}s)}{2}
    &< \frac 18 A(s).
    \label{phi-3}
\end{align}
Then \eqref{R1, phi} and \eqref{phi-3} imply that for any $s\in [t, t+\ve^2],$ 
\begin{align}
    \nl{\im \phi(\ve^{-2}s)}{2}
    &< \frac 18 A(s).
    \label{phi-1}
\end{align}
This together with \eqref{phi-0} gives that for any $s\in [t, t+\ve^2],$
\begin{align}
    \nl{\re \phi(\ve^{-2}s)}{2}
    &> \frac 78 A(s).
    \label{phi-2}
\end{align}
By the Mean Value Theorem, we have that 
\begin{align*}
    R_1(\ve^{-2}t +1)
    =R_1 (\ve^{-2}t)
    + (\p_t R_1) (\ve^{-2}\ti t),
\end{align*}
where $\ti t \in [t, t+\ve^2]$.
In view of this, \eqref{R1, phi}, \eqref{phi-1}, and \eqref{phi-2},
we obtain
\begin{align*}
    \nl{ R_1(\ve^{-2}t +1)}{2}
    &\ge 
    \nl{(\p_t R_1) (\ve^{-2}\ti t)}{2}
    -\nl{R_1 (\ve^{-2}t)}{2}\\ 
    &> 
    \frac 78 A(\ti t)
    -\frac 18 A(t)
    \\
    &>
    \frac 34 A(t+\ve^2)
        -O(\ve^2)A'(t+\ve^2)
        -O(\ve^4)A''(t+\ve^2)
        -O(\ve^6)A'''(t+\ve^2).
\end{align*}
Therefore, for $\ve $ sufficiently small, we get
\begin{align*}
   \nl{ R_1(\ve^{-2}(t +\ve^2))}{2}
   = 
   \nl{ R_1(\ve^{-2}t +1)}{2}
   &\ge \frac 18 A(t+\ve^2).
\end{align*}
Thus \eqref{phi-0-0} holds for $s=t+\ve^2$, which contradicts with \eqref{phi-3}. 
Hence, the first claim is proved.

Secondly, we prove that \eqref{phi-0-0} holds for all $t\in  I$, based on \eqref{phi-0} for all $t\in I$.
For any $t \in [1+\ve^2, \ve^{-2}\de_0 ]$, 
 we have the interval $[t-\ve^2, t]\subset I$ where \eqref{phi-0} holds.
Then by the first claim, we have for any $t \in [1+\ve^2, \ve^{-2}\de_0 ]$, 
\begin{align*}
   \nl{R_1(\ve^{-2}t)}{2}
   &\ge A(t).
\end{align*}
Let $\ve$ tend to 0, then we finish Step 2.                   

The desired conclusion in lemma is a consequence of Step 1 and Step 2.
\end{proof}

\subsection{The proof of Theorem \ref{optimal rate}}
The proof of Theorem \ref{optimal rate} is divided into two cases:
the regular case(smooth initial data) and the non-regular case($H^\al_x$-data).
\subsubsection{Regular case}
Let $v_0=\de_0 g(x)$ and  
\begin{align*}
   g(x):=A_0^{\frac d2}e^{-A_0^2|x|^2} \in \mathscr{S}(\R^d), 
\end{align*}
where $A_0$ is a suitable large constant and will be chosen later.
Then we have for any $\ga \ge 0,$
\begin{align*}
   \nh{g}{\ga}=C A_0^\ga,  
\end{align*}
where the constant $C$ is independent of $A_0$. 
Due to Lemma \ref{lem:est of phi_de},
we obtain that for any $t\in  I$,
   \begin{align*}
       \nl{r_1(x)}{2}
       &\ge 
       A_0^8\ve^4 t^2\de_0
       \big[
        C_1
        +C_2A_0^4\ve^2t-C_0A_0^{-2}t^{-1}
        -\ti CA_0^{-4}t^{-2}
         -\ti C \de_0^2 (A_0^{-4}t^{-2}
                +A_0^{-2}t^{-1}+1)
        \\
   &\quad -\ti C \ve^2 (A_0^{-2}t^{-2}+t^{-1}+A_0^2)
        -\ti C \ve^4(t^{-2}+A_0^4+A_0^6t)        
       \big]
   \end{align*}
   where the constants $C_0, C_1, C_2$, and $\ti C$ are independent of $t$, $\de_0$, and $\ve$.  
Choosing $\ve $ sufficiently small, $\de_0 $ suitably small, $A_0$ suitably large,
we have for any $t\in  I, $
\begin{align*}
    \nl{r_1(t)}{2}
    &\ge  \frac 12 C_1 A_0^8\ve^{4}\de_0 t^2.
    \label{phi1}
\end{align*}
This finishes the proof of the regular case.

\subsubsection{Non-regular case}
In this subsection, we address the non-regular case.
Let $v_0=\de_0 g(x)$ and
\begin{align*}
   g(x)=\F^{-1} \Big(
        \jb{\xi}_2^{-\al-\frac d2}\left(\ln \jb{\xi}_2\right)^{-1}
   \Big), 
\end{align*}
where $\jb{\cdot}_2 =\sqrt{|\cdot|^2+2}$.
This implies that $g\in H^\al_x(\R^d)$.
Moreover, we have the following estimates:
\begin{equation}
    \begin{aligned}
       \norm{g}_{H^\al_x}
       &\sim 1,  
       \\
       \norm{P_{\le N }g}_{H^\ga_x}
       &\sim N^{ \ga-\al}\left(\ln N\right)^{-1},
       \quad \ga>\al,
       \\
       \norm{P_{> N }g}_{H^\ga_x}
       &\les N^{\ga-\al }\left(\ln N\right)^{-1},
       \quad 
       \ga < \al.
    \end{aligned}
    \label{value of H al}
\end{equation}
Recall the definition of $r_1$, we denote
\begin{align*}
   r_1 = \left(u^\ve - {\Phi_1}_N-\ve^2{\Phi_2}_N\right)
            -\left(\Phi_1- {\Phi_1}_N+ \ve^2\Phi_2-\ve^2{ \Phi_2}_N\right)
       =:{r_1}_N -{r^N_1},
\end{align*}
where ${\Phi_1}_N$ and ${\Phi_2}_N$ are given in \eqref{w1,w2}.
Arguing similarly as in Section \ref{the nonregular case},
 we establish
 \begin{align}
    \nl{{r_1^N}}{2}
    &\les  \de_0\left(
        N^{-\al}
        +\ve^2N^{2-\al}
        +\ve^2tN^{4-\al}
    \right)(\ln N)^{-1}.
    \label{al 3}
 \end{align}
To estimate $\nl{{r_1}_N}{2}$, 
we imitate the proof of Lemma \ref{lem:est of phi_de} and obtain a similar conclusion with $g$ replaced by $P_{\le N}g$ for $g\in H^\al$, under the condition given in \eqref{assumption of T(scaling back)} that for small $\de_1$,   $t$ satisfies 
\begin{align} 
\label{de_0-1}
    \ve^2\de_0tN^{s_c}(\ln N)^{-1}     
     \les \de_1 
     \mbox{ when }
     \al \in [4, 4+s_c),
    \mbox{ or }
    \ve^2\de_0 t(\ln N)^{-1}     
     \les \de_1
      \mbox{ when }
    \al \in [4+s_c, 8].
\end{align}
Then applying the result of Lemma \ref{lem:est of phi_de} with $g$ replaced by $P_{\le N}g$, \eqref{value of H al}, and \eqref{al 3}, 
we have for any $t\in  I$ and $g\in H^\al$ with $\al\in[4,8)$,
 \begin{align}
     \nl{{r_1}}{2}
     &\ge \ve^4t^2N^{8-\al}\de_0 
            (\ln N)^{-1}
            \big[
                C_1+C_2\ve^2tN^4
                -C_0t^{-1}N^{\al -8}\eta_\al (N)
                -Ct^{-2}N^{\al -8}\ln N
            \notag \\
          &\quad
          -C\de_0^2(t^{-2}N^{\al -8}\ln N
          +t^{-1}\eta_\al (N) N^{\al-8}+1)        
          -C\ve^2(\eta_\al (N) N^{\al-8}t^{-2}+t^{-1}+N^2)              
             \notag  \\
          &\quad
          -C\ve^4(t^{-2}+N^4+tN^6)
          -C\ve^{-4}t^{-2}
            (N^{-8} +\ve^2N^{-6}+\ve^2tN^{-4})
            \big],
          \label{al 1'}
 \end{align}
 where $\eta_\al (N)= N^{6-\al}$ for $\al \in [4,6)$, and $\ln N$ for $\al \in [6,8)$.
 And for any $t\in  I$ and $g\in H^8$, we have 
 \begin{align}
     \nl{{r_1}}{2}
     &\ge \ve^4t^2\de_0
     \big[
        C_1+C_2\ve^2tN^4 (\ln N)^{-1}-C_0t^{-1}
        -Ct^{-2}-C\de_0^2t^{-2}(1+t+t^2)
       \notag  \\
    &\quad
        -C\ve^2t^{-2}(1+t+t^2N^2(\ln N)^{-1})
        -C\ve^4t^{-2}(1+t^2N^4(\ln N)^{-1}
            +t^3N^6(\ln N)^{-1}) 
       \notag  \\
    &\quad
        -C\ve^{-4}t^{-2}
            (N^{-8} +\ve^2N^{-6}+\ve^2tN^{-4})
            (\ln N)^{-1}
     \big]
     \label{al 2}
 \end{align}

We divide the range of $\al$ into the intervals $[4, 6)$ and $[6,8]$.
 
We consider $ \al \in [4, 6)$ and $t\in  I$ first.
 Since $t \in  I$ implies $t\ge 1$,
 it suffices to consider the case where $(\ve, N, t)$ satisfies \textbf{Subcase $\text{A}_2$} in Subsection \ref{proof of main theorem}, namely
 \begin{align*}
     1\ges \ve^2N^2, \ve^2tN^4 \text{, and } tN^2\ges 1, 
 \end{align*}
 or \textbf{Subcase $\text{A}_6$}:
 \begin{align*}
    \ve^2 tN^4 \ges 1, \ve^2N^2 \text {, and } tN^2\ges 1. 
 \end{align*}
Here we only prove the first case as the second one is similar.
Suppose that the first case holds, 
       then \eqref{al 1'} yields that 
        \begin{align*}
            \nl{r_1}{2}
            &\ge
            \ve^4t^2N^{8-\al}\de_0 
            (\ln N)^{-1}
            \big[
                C_1+C_2\ve^2tN^4
                -C_0t^{-1}N^{-2}
                -Ct^{-2}N^{\al -8}\ln N
                 \\
          &\quad
          -C\de_0^2(t^{-2}N^{\al -8}\ln N
          +t^{-1}N^{-2}+1)        
          -C\ve^2(N^{-2}t^{-2}+t^{-1}+N^2)              
               \\
          &\quad
          -C\ve^4(t^{-2}+N^4+tN^6)
          -C\ve^{-4}t^{-2}N^{-8}
            \big].
        \end{align*}
        Choosing(for both cases) $$N= A_0\left(\ve^2t\right)^{-\frac 14},$$
        where $A_0$ is a suitable large constant, then $t\in I $ implies that $t$ satisfies \eqref{de_0-1}.
        The above inequality for $r_1$ becomes
        \begin{align*}
            \nl{r_1}{2}
           &\ge A_0^{8-\al}\de_0
           \left(\ve^2 t\right)^{\frac \al 4}
           \big|
             \ln \big(A_0^{-4}  \ve^2t \big)
             \big|^{-1}  
           \bigg\{
                   \left(
              \ti C_1 
             + \ti C_2 A_0^{4}
             \right) 
            -\ti C_0A_0^{-2}(\ve t^{-\frac 12})
              \\&
            \qquad
            -\ti C(1+\de_0^2)
                 \big|
                     \ln \big(A_0^{-4}  \ve^2t \big)
                 \big|
             A_0^{\al -8}
            \ve^{\frac{8-\al}{2}}t^{-\frac \al 4}   
            -\ti C \de_0^2 
             \left(
             A_0^{-2}\ve t^{-\frac 12}
             +1
             \right)
            \\&
            \qquad
            -\ti C
             \left(
             A_0^{-2}\ve^3t^{-\frac 32}
             +\ve^2t^{-1}
             +A_0^{2} \ve t^{-\frac 12}
             \right)
             \\
             &\qquad
             -\ti C 
                \left(
                    \ve^4t^{-2}
                    +A_0^{4}\ve^2t^{-1}
                    +A_0^{6}\ve t^{-\frac 12}
                \right)            
             -\ti C A_0^{-8}
           \bigg\}, 
        \end{align*}
        where the positive constants $\ti C, \ti C_0, \ti C_1$, and $\ti C_2$ are independent of $t, \de_0$, and $\ve$.
        Choosing $\de_0 $ suitably small, $A_0$ suitably large,
        we establish 
        \begin{align*}
           \nl{r_1}{2}
           &\ge \frac 12 A_0^{8-\al} \de_0
           \left(
           \ti  C_1 
           +\ti C_2 A_0^{4}
           \right) \left(\ve^2t\right)^{\frac \al 4}
           \big|
           \ln \big(A_0^{-4}  \ve^2t \big)
           \big|^{-1}. 
        \end{align*}
        This finishes the proof of the case $\al \in [4, 6)$.
        
       Similarly, for $\al \in [6, 8]$, it suffices to consider the case where $(\ve, N, t)$ satisfies
       \textbf{Subcase $\text{B}_1$} in Subsection \ref{proof of main theorem},
          \begin{align*}
              1\ges \ve^2N^2, \ve^2tN^4, 
          \end{align*}
          or \textbf{Subcase $\text{B}_3$}
          \begin{align*}
              \ve^2 tN^4 \ges 1, \ve^2N^2.
        \end{align*}
          Choosing the same $N$ as in the case $\al \in [4, 6)$,
           then $t\in I $ implies that $t$ satisfies \eqref{de_0-1}.
          Arguing similarly as the case $\al \in [4, 6)$, and
          using \eqref{al 1'} when $\al \in [6,8)$, or \eqref{al 2} when $\al =8$,
           an analogous conclusion holds for $\al \in [6, 8]$.
        
       Hence, combining results for $\al \in [4,6) $ and $\al\in[6,8]$, we complete the proof of the non-regular case.
       Thus we prove Theorem \ref{optimal rate}.

\noindent {\bf Acknowledgments}
The authors would like to thank Professor Yifei Wu for helpful discussions and comments.
The authors are partially supported by NSFC 12171356.

\bibliography{ref}
\bibliographystyle{abbrv}

\end{document}